\documentclass[12pt]{article}

\usepackage{fullpage}
\usepackage{amsmath}
\usepackage{amssymb}
\usepackage{amsthm}
\usepackage{graphicx}
\usepackage{placeins}
\usepackage{extpfeil}
\usepackage{enumerate}
\usepackage{caption}
\usepackage{subcaption}
\usepackage{hyperref}
\usepackage{xcolor}
\usepackage{pbox}
\usepackage[numbers,sort&compress]{natbib}

\newtheorem{thm}{Theorem}[section]
\newtheorem{theorem}{Theorem}

\newtheorem{rmk}{Remark}[section]
\newtheorem{lem}{Lemma}[section]

\newtheorem{example}{Example}[section]

\newcommand{\dive}{{\rm div \hspace{0.05cm} }}
\newcommand{\Real}{{\rm Re \hspace{0.05cm} }}
\newcommand{\Imagine}{{\rm Im \hspace{0.05cm} }}
\newcommand*{\avint}{\mathop{\ooalign{$\int$\cr$-$}}}
\newcommand{\R}{\mathbb{R}}
\newcommand{\bS}{\mathbb{S}}

\begin{document}

\title{Removable singularity of $(-1)$-homogeneous solutions of stationary Navier-Stokes equations}
\author{
	Li Li\footnote{School of Mathematics and Statistics, Ningbo University, 818 Fenghua Road, Ningbo, Zhejiang 315211, China. Email: lili2@nbu.edu.cn. },  
	YanYan Li\footnote{Department of Mathematics, Rutgers University, 110 Frelinghuysen Road, Piscataway, NJ 08854, USA. Email: yyli@math.rutgers.edu. }, 
	Xukai Yan\footnote{Department of Mathematics, Oklahoma State University, 401 Mathematical Sciences Building, Stillwater, OK 74078, USA. Email: xuyan@okstate.edu. }}
\date{}
\maketitle

\abstract{
	We study the removable singularity problem for $(-1)$-homogeneous solutions of the three-dimensional incompressible stationary Navier-Stokes equations with singular rays. We prove that any local $(-1)$-homogeneous solution $u$ near a potential singular ray from the origin, which passes through a point $P$ on the unit sphere $\bS^2$, can be smoothly extended across $P$ on $\bS^2$, provided that $u=o(\ln \text{dist} (x, P))$ on $\bS^2$. The result is optimal in the sense that for any $\alpha>0$, there exists a local $(-1)$-homogeneous solution near $P$ on $\bS^2$, such that $\lim_{x\in \mathbb{S}^2, x\to P}|u(x)|/\ln \text{dist}(x, P)=-\alpha$. Furthermore, we discuss the behavior of isolated singularities of $(-1)$-homogeneous solutions and provide examples from the literature that exhibit varying behaviors. We also present an existence result of solutions with any finite number of singular points located anywhere on $\bS^2$. 
}

\section{Introduction}\label{sec:intro}

Consider the three-dimensional incompressible stationary Navier-Stokes equations, 
\begin{equation}\label{eq:NS}
	\left\{
	\begin{aligned}
		& -\Delta u + ( u \cdot \nabla ) u + \nabla p = 0, \\ 
		& \dive u=0,
	\end{aligned}
	\right.
\end{equation}
where $u: \mathbb{R}^3\to\mathbb{R}^3$ is the velocity vector and $p:\mathbb{R}^3\to\mathbb{R}$ is the pressure. These equations are invariant under the scaling $u(x)\to \lambda u(\lambda x)$ and $p(x)\to \lambda^2 p(\lambda x)$ for any $\lambda>0$. It is natural to study solutions which are invariant under this scaling. For such solutions, $u$ is $(-1)$-homogeneous and $p$ is $(-2)$-homogeneous, and we call them $(-1)$-homogeneous solutions according to the homogeneity of $u$. In general, a function $f$ is said to be $(-k)$-homogeneous if $f(x)= \lambda^k f(\lambda x)$ for any $\lambda>0$. 

Let $x=(x_1,x_2,x_3)$ be the Euclidean coordinates and $e_1=(1,0,0)$, $e_2=(0,1,0)$, $e_3=(0,0,1)$ be the standard basis. In this paper, we denote $x'=(x_1,x_2)$ and $\nabla'=(\partial_1, \partial_2)$. Let $(r,\theta, \phi)$ be the spherical coordinates, where $r$ is the radial distance from the origin, $\theta$ is the angle between the radial vector and the positive $x_3$-axis, and $\phi$ is the meridian angle about the $x_3$-axis. A vector field $u$ can be written as
\[
	u = u_r e_r + u_\theta e_{\theta} + u_\phi e_{\phi},
\]
where
\[
	e_r = \left(
	\begin{matrix}
		\sin\theta\cos\phi \\
		\sin\theta\sin\phi \\
		\cos\theta
	\end{matrix} \right), \hspace{0.5cm}
	e_{\theta} = \left(
	\begin{matrix}
		\cos\theta\cos\phi \\
		\cos\theta\sin\phi \\
		-\sin\theta	
	\end{matrix} \right), \hspace{0.5cm}
	e_{\phi} = \left(
	\begin{matrix}
		-\sin\phi \\ \cos\phi \\ 0
	\end{matrix} \right).
\]
A vector field $u$ is \emph{axisymmetric} if $u_r$, $u_{\theta}$ and $u_{\phi}$ are independent of $\phi$, and \emph{no-swirl} if $u_{\phi}\equiv 0$. 

In 1944, Landau \cite{Landau} discovered a 3-parameter family of explicit $(-1)$-homogeneous solutions of the stationary Navier-Stokes equations in $C^\infty(\mathbb{R}^3\setminus\{0\})$, see also \cite{Slezkin} and \cite{Squire}. These solutions, now called \emph{Landau solutions}, are axisymmetric with no swirl and have exactly one point singularity at the origin. Tian and Xin proved in \cite{TianXin} that all $(-1)$-homogeneous, axisymmetric nonzero solutions of (\ref{eq:NS}) in $C^\infty(\mathbb{R}^3\setminus\{0\})$ are Landau solutions. \v{S}ver\'{a}k established the following result in 2006:
\begin{theorem}[\cite{Sverak}]\label{thm:Sverak}
	All $(-1)$-homogeneous nonzero solutions of (\ref{eq:NS}) in $C^2(\mathbb{R}^3\setminus\{0\})$ are Landau solutions.
\end{theorem}
He also proved in the same paper that there is no nonzero $(-1)$-homogeneous solution of the incompressible stationary Navier-Stokes equations in $C^{2}(\mathbb{R}^n\setminus\{0\})$ for $n\ge 4$. In dimension two, he characterized all such solutions satisfying a zero flux condition. Homogeneous solutions of (\ref{eq:NS}) have been studied in other works as well, see \cite{CKPW, Goldshtik, Gusarov, JLL, KT, LZZ, PP1, PP2, PP3, Serrin, Slezkin, Squire, Wang, Yatseyev, ZZ}. There have also been works on homogeneous solutions of Euler's equations, see \cite{Abe, LS, Shvydkoy} and the references therein. 

For $(-1)$-homogeneous solutions $(u, p)$ in $\mathbb{R}^3\setminus\{0\}$, (\ref{eq:NS}) can be reduced to a system of partial differential equations of $(u, p)$ on $\mathbb{S}^2$. For any set $\Omega\subset\bS^2$, a $(-1)$-homogeneous solution $(u, p)$ on $\Omega$ is understood to have been extended to the set $\{x\in \mathbb{R}^3\mid x/|x|\in \Omega\}$ so that $u$ is $(-1)$-homogeneous and $p$ is $(-2)$-homogeneous. We use this convention throughout the paper unless otherwise stated. 

Theorem \ref{thm:Sverak} has classified all $(-1)$-homogeneous solutions of (\ref{eq:NS}) in $C^2(\mathbb{S}^2)$. A natural next step is to study $(-1)$-homogeneous solutions of (\ref{eq:NS}) in $C^{2}(\mathbb{S}^2\setminus\{P_1,..., P_m\})$ for finitely many points $P_1,..., P_m$ on $\mathbb{S}^2$.  In \cite{LLY1}-\cite{LY}, we 
 studied $(-1)$-homogeneous axisymmetric solutions of (\ref{eq:NS}) in $C^2(\mathbb{S}^2\setminus\{S, N\})$, where $S$ is the south pole and $N$ is the north pole. In \cite{LLY1},  all $(-1)$-homogeneous axisymmetric no-swirl solutions in $C^2(\mathbb{S}^2\setminus\{S\})$ were classified, and the existence and non-existence results of $(-1)$-homogeneous axisymmetric solutions with nonzero swirl in $C^2(\mathbb{S}^2\setminus\{S\})$ were established. The asymptotic expansions of all local $(-1)$-homogeneous axisymmetric solutions of (\ref{eq:NS}) near a singular ray were also derived in \cite{LLY1}. In \cite{LLY2},  all $(-1)$-homogeneous axisymmetric no-swirl solutions in $C^2(\mathbb{S}^2\setminus\{S, N\})$ were classified. In \cite{LLY3},  the existence and non-existence results for $(-1)$-homogeneous axisymmetric solutions in $C^2(\mathbb{S}^2\setminus\{S, N\})$ with nonzero swirl were established. In \cite{LLY4}, the vanishing viscosity limit of $(-1)$-homogeneous axisymmetric no-swirl solutions of (\ref{eq:NS}) in $C^2(\mathbb{S}^2\setminus\{S, N\})$ was studied. In \cite{LY}, the asymptotic stability of the least singular homogeneous axisymmetric no-swirl solutions under $L^2$-perturbations was proved. Note that the asymptotic stability of Landau solutions under $L^2$-perturbations was proved by Karch and Pilarczyk in \cite{KP}. 

To study the $(-1)$-homogeneous solutions of (\ref{eq:NS}) with finite singularities $P_1, ..., P_m$ on $\bS^2$, it is helpful to first analyze the behavior of solutions near an isolated singularity on $\bS^2$. This paper studies the following removable singularity problem: For local $(-1)$-homogeneous solutions (not necessarily axisymmetric) of (\ref{eq:NS}) near a potential singular ray, under what condition the singular ray is removable? Namely, under what condition the solution can be smoothly extended across the singular ray except a possible singularity at the origin? 

There has been much study on the behavior of solutions of (\ref{eq:NS}) near isolated singularities in $\R^3$, see e.g. \cite{DE70, Shapiro1, Shapiro2, Shapiro3, CK00, KK06, KS11, MT12, NP00, Sverak}. 

Without loss of generality, we consider local $(-1)$-homogeneous solutions of (\ref{eq:NS}) near a potential singular ray from the origin passing through the south pole $S$. It is equivalent to studying the solutions in a small neighborhood of $S$ on $\mathbb{S}^2$. The asymptotic expansions of axisymmetric solutions obtained in \cite{LLY1} suggest that the least singular behavior near a singular ray through $S$ is in the order of $\ln {\rm dist}(x,S)$. Therefore, a natural removable singularity condition is $u=o(\ln {\rm dist}(x,S))$. Denote $B_{\delta}(S) := \{ x \in \mathbb{R}^3 \mid {\rm dist}(x,S)<\delta\}$ for $\delta>0$. Clearly, ${\rm dist}(x,S)/|x'|\to 1$ as $x\to S$ on $\bS^2$. We have the following result.  

\begin{thm}\label{thm_main}
	Let $\delta>0$, 
	$(u, p)\in C^{2}(\mathbb{S}^2\cap B_{\delta}(S)\setminus\{S\})$ be a $(-1)$-homogeneous solution of (\ref{eq:NS}). If 
	\begin{equation}\label{eq:thm:main}
		\lim_{x\in \mathbb{S}^2, x\to S}\frac{|u(x)|}{\ln {\rm dist}(x,S)}=0. 
	\end{equation}
	 Then $(u, p)$ can be extended as a $C^2$ function in $\mathbb{S}^2\cap B_{\delta}(S)$. 
\end{thm}

\begin{rmk}
	The above removable singularity result is optimal in the following sense: For any $\alpha>0$, there exists a $(-1)$-homogeneous axisymmetric no-swirl solution $(u, p)\in C^{\infty}(\mathbb{S}^2\setminus\{S, N\})$ of (\ref{eq:NS}), such that $\lim_{x\in \mathbb{S}^2, x\to S}|u(x)|/\ln {\rm dist}(x, S)=$\\$\lim_{x\in \mathbb{S}^2, x\to N}|u(x)|/\ln {\rm dist}(x, N) = - \alpha$.
	Examples of such solutions can be found in \cite{LLY2}, see also Example \ref{example:least:sing} in Section \ref{sec:3}. 
	On the other hand, 
	there does not exist $(-1)$-homogeneous axisymmetric solution $u\in C^2(\mathbb{S}^2\setminus\{P\})$ of (\ref{eq:NS}) satisfying $0<\limsup_{|x|=1,x\to P}|u(x)|/|\ln {\rm dist}(x, P)|<\infty$, where $P=S$ or $N$, see Lemma \ref{lem:3_1}. 
\end{rmk}

Note that all $(-1)$-homogeneous axisymmetric solutions of (\ref{eq:NS}) in $C^{2}(\bS^2\setminus\{S, N\})$ satisfying $u=O(\ln |x'|)$ as $x\to S$ or $N$ must have no swirl. see Lemma \ref{lem:3_1}. 

The organization of the paper is as follows: Theorem \ref{thm_main} is proved in Section \ref{sec:2}. In Section \ref{sec:3}, we provide further discussion on the behavior of isolated singularities of $(-1)$-homogeneous solutions of equation (\ref{eq:NS}). Specifically, we describe the asymptotic behavior of $(-1)$-homogeneous axisymmetric solutions of equation (\ref{eq:NS}) obtained in \cite{LLY1}-\cite{LLY3}. Additionally, we list and discuss several special examples of $(-1)$-homogeneous solutions of equation (\ref{eq:NS}) from the literature. In Theorem \ref{thm:3_1}, we also present an existence result on $(-1)$-homogeneous solutions of equation (\ref{eq:NS}) that have exactly $m$ singularities on $\bS^2$, where $m\ge 2$. 

\medskip

\noindent
{\bf Acknowledgment}. The work of the first named author is partially supported by NSFC grant 12271276 and ZJNSF grant LR24A010001. The work of the second named author is partially supported by NSF grant DMS-2000261, DMS-2247410, and Simons Fellows Award 677077. The work of the third named author is partially supported by 
Simons Foundation Travel Support for Mathematicians 962527 and NSF Career Award DMS-2441137.

\section{Proof of Theorem \ref{thm_main}}\label{sec:2}

For real numbers $a<b$ and $R>0$, denote
\begin{equation}
	\Omega_{a, b, R} := \{x\in \mathbb{R}^3 \mid a<x_3<b, |x'|<R\}. 
\end{equation}

\begin{lem}\label{lem:2_1}
	For $M, R>0$ and $a<b$ satisfying $ab>0$, let $(u,p)$ be a $C^{2}$ solution of 
	\begin{equation}
	\left\{ 
	\begin{aligned}
		& - \Delta u + \nabla p = \lambda u \cdot \nabla u, \quad && x\in\Omega_{a, b, R}\setminus\{x'=0\},\\
		& \dive u=0, \quad && x\in \Omega_{a, b, R}\setminus\{x'=0\},
	\end{aligned}
	\right.
	\end{equation}
	satisfying $u=o(\ln |x'|)$ as $x'\to 0$, uniform for $x\in \Omega_{a, b, R}\setminus\{x'=0\}$ and $\lambda\in [-M, M]$. Then $u$ is smooth in $\Omega_{a, b, R}\setminus \{x'=0\}$, and for any $a<a'<b'<b$, $0<R'<R$, and any integer $k\ge 1$, 
	\begin{equation}\label{eq:lem:2_1}
		| \nabla^k u | = o \left(\frac{\ln |x'|}{|x'|^k}\right), 
		\quad |\nabla^{k-1} p|=o\left(\frac{\ln |x'|}{|x'|^{k}}\right), \quad \textrm{as } x'\to 0, 
	\end{equation}
	uniform for $x\in \Omega_{a', b', R'}\setminus\{x'=0\}$ and $\lambda\in [-M, M]$. 
\end{lem}
\begin{proof}
	The smoothness of $u$ in $\Omega_{a, b, R}\setminus\{x'=0\}$ follows from a bootstrap argument using standard estimates for Stokes equations. Now we prove (\ref{eq:lem:2_1}). Given $a<a'<b'<b$, let $\bar{x} \in \Omega_{a', b', R} \setminus\{x'=0\}$ be an arbitrary fixed point and $\bar{\rho}:=\min\{|\bar{x}'|/3, (R-|\bar{x}'|)/3, (b-b')/3, (a'-a)/3\}$. Define $\bar{u}: B_2\to\mathbb{R}^3$ and $\bar{p}: B_2\to\mathbb{R}$ by 
	\begin{equation*}
		\bar{u}(y) := \frac{u(\bar{x}+\bar{\rho}y)}{\ln \bar{\rho}}, \quad 
		\bar{p}(y) := \frac{\bar{\rho}}{\ln \bar{\rho}}p(\bar{x}+\bar{\rho}y), \quad y\in B_2,
	\end{equation*}
	where $B_2=B_2(0)\subset \mathbb{R}^3$ is the ball of radius $2$ centered at $0$. Then 
	\begin{equation}\label{eq:lem:2_1:pf:1}
		-\Delta \bar{u}+\nabla\bar{p}
		=(\lambda\bar{\rho}\ln \bar{\rho}) \bar{u}\nabla\bar{u}=:f, \quad \textrm{ in } B_2,
	\end{equation}
	and
	\begin{equation}\label{eq:lem:2_1:pf:2}
		\dive \bar{u}=0, \quad \textrm{ in }B_2.
	\end{equation} 
	By the assumption on $u$, we have
	\begin{equation}\label{eq:lem:2_1:pf:3}
		\sup_{B_2}|\bar{u}|=o(1), \quad \textrm{ as }\bar{\rho}\to 0. 
	\end{equation}
	So for any $1<s<\infty$, 
	\begin{equation}\label{eq:lem:2_1:pf:4}
		\| f \|_{W^{-1, s}(B_2)}\le |\lambda \bar{\rho}\ln \bar{\rho}| \| |\bar{u}|^2 \|_{L^s(B_2)}=o(1),
		\quad \textrm{as } \bar{\rho}\to 0. 
	\end{equation}
	Note that the convergence rates in this proof are uniform for $|\lambda|\le M$. 
	By (\ref{eq:lem:2_1:pf:1}), (\ref{eq:lem:2_1:pf:2}), (\ref{eq:lem:2_1:pf:3}) and (\ref{eq:lem:2_1:pf:4}), using interior estimates of the Stokes equations (see Theorem 2.2 in \cite{SverakTsai}), 
	we have, for any $1<s<\infty$ and $0<r<2$, that 
	\begin{equation}\label{eq:lem:2_1:pf:5}
		\| \bar{u} \|_{W^{1, s}(B_r)} + \inf_{c\in\mathbb{R}}\|\bar{p}-c\|_{L^s(B_r)} = o(1), \quad \textrm{as } \bar{\rho}\to 0. 
	\end{equation}
	By (\ref{eq:lem:2_1:pf:3}) and (\ref{eq:lem:2_1:pf:5}), we have $\|f\|_{L^s(B_r)}=o(1)$. Then by estimates for the Stokes equations (see e.g. Theorem IV.4.1 in \cite{Galdi}) and Sobolev embedding theorems, we have 
	\begin{equation}\label{eq:lem:2_1:pf:6}
		\|\bar u\|_{W^{2, s}(B_{r})}+\|\nabla \bar p\|_{L^{s}(B_{r})}=o(1) \textrm{ and } \|\bar{u}\|_{C^1(B_r)}=o(1)
	\end{equation}
	for any $0<r<2$. It follows that $\|f\|_{W^{1, s}(B_r)}=o(1)$ for any $1<s<\infty$. 

	By estimates for the Stokes equations and Poincar\'{e}'s inequality, we have, for any $l\ge 2$ and $0<r<r'<2$, that
	\begin{equation}\label{eq:lem:2_1:pf:7}
	\begin{split}
		\|\bar{u}\|_{W^{l, s}(B_r)}+\|\nabla \bar{p}\|_{W^{l-2, s}(B_r)} & \le C(\|f\|_{W^{l-2, s}(B_{r'})}+\|\bar{u}\|_{W^{1, s}(B_{r'})}+\|\bar{p}-\avint_{B_r}\bar{p}\|_{L^s(B_{r'})}) \\
		& \le C(\|f\|_{W^{l-2, s}(B_{r'})}+\|\bar{u}\|_{W^{1, s}(B_{r'})}+\|\nabla \bar{p}\|_{L^s(B_{r'})}) 
	\end{split}
	\end{equation}
	for some $C$ depending only on $s$, $r$, $r'$ and $l$. With a standard bootstrap argument using Sobolev embedding theorems, by (\ref{eq:lem:2_1:pf:6}) and (\ref{eq:lem:2_1:pf:7}) we have that 
	\[
		\|\bar{u}\|_{W^{l, s}(B_r)}+\|\nabla \bar{p}\|_{W^{l-2, s}(B_r)}=o(1), 
	\]
	for any $l\ge 3$ and $0<r<2$. Then by Sobolev embedding theorems, we have 
	\begin{equation*}
		\sup_{B_{r}}(|\nabla^k \bar{u}|+|\nabla^k\bar{p}|)=o(1), \quad \forall k\ge 1. 
	\end{equation*}
	So (\ref{eq:lem:2_1}) holds for all $k\ge 2$. For $k=1$, the first estimate in (\ref{eq:lem:2_1}) holds from the above. 

	Now we prove the second estimate in (\ref{eq:lem:2_1}) when $k=1$. We only need to prove it when $x_2=0$ and $x_1>0$. For any $0<x_1<R/2$, we have 
	\[
		p(x_1, 0, x_3) =p(1, 0, x_3)-\int_{x_1}^{1}\partial_1p(t, 0, x_3)dt. 
	\]
	Since $p\in C^{\infty}(\Omega_{a', b', R}\setminus\{x'=0\})$, we have $p(1, 0, x_3)=O(1)=o(\ln |x_1|/|x_1|)$. So we only need to show 
	\begin{equation}\label{eq:lem:2_1:pf:8}
		\lim_{x_1\to 0}\frac{\int_{x_1}^{1}\partial_1 p(t, 0, x_3) dt }{\big| \ln |x_1| \big| / |x_1|}=0. 
	\end{equation}
	To see this, note that we have proved (\ref{eq:lem:2_1}) holds for all $k\ge 2$, thus $\nabla p(x)=o(\ln |x'|/|x'|^2)$. So for any $\epsilon>0$, there exists some $\delta>0$, such that $|\nabla p|\le \epsilon | \ln |x'| | / |x'|^2$ for any $0<|x'|<\delta$. So 
	\[
	\begin{split}
		\frac{|\int_{x_1}^{1}\partial_1p(t, 0, x_3)dt| }{| \ln |x_1| |/|x_1|} 
		& \le \frac{|x_1|}{| \ln |x_1| |}\left(\int_{x_1}^{\delta}|\nabla p(t, 0, x_3)|+\int_{\delta}^{1}|\nabla p(t, 0, x_3)|\right) \\
		& \le \epsilon \frac{|x_1|}{| \ln |x_1| |}\int_{x_1}^{\delta} \frac{| \ln t |}{ t^2}dt+C\frac{|x_1|}{| \ln |x_1| |} \\
		& \le C\epsilon, \quad \textrm{ as }x_1\to 0, 
	\end{split}
	\]
	for some $C$ depending only on $\delta$. So (\ref{eq:lem:2_1:pf:8}) holds. 
	The lemma is proved. 
\end{proof}

Denote $D_R:=\{x'\in\mathbb{R}^2 \mid |x'|<R\}$ for any $R>0$. 

\begin{lem}\label{lem:2_2}
	Let $R>0$, 
	$a<b$ satisfying $ab>0$,
	and $F\in C^{\infty}(\bar{\Omega}_{a, b, R}\setminus\{x'=0\})$ be a $(-3)$-homogeneous vector-valued function.
	Suppose $q\in C^{\infty}(\bar{\Omega}_{a, b, R}\setminus\{x'=0\})$ is a $(-2)$-homogeneous function satisfying 
	\begin{equation*}
		-\Delta q=\dive F(x), 
		\quad \textrm{ in }\Omega_{a, b, R}\setminus\{x'=0\}. 
	\end{equation*}
	Assume there exists some $\delta\in (0, 2)$ such that 
	\begin{equation}\label{eq:lem:2_2:1}
		|q(x)| |x'|^{2-\delta}+\sum_{j=0}^{2}|\nabla^jF| |x'|^{j+2-\delta}=o(1)
	\end{equation}
	as $x'\to 0$ uniformly in $\Omega_{a, b, R}\setminus\{x'=0\}$. Then there exist some $h(x) \in L^{\infty}_{loc}((a, b), W^{1, s}(D_R))$ for any $1< s<\frac{2}{2-\delta}$, and $a_0(x_3), a_1(x_3), b_1(x_3)\in C(a, b)$, such that 
	\begin{equation}\label{eq:lem:2_2:2}
		q(x)=h(x)+a_0(x_3)\ln |x'|+a_1(x_3)\frac{x_1}{|x'|^2}+b_1(x_3)\frac{x_2}{|x'|^2}, \quad \textrm{ in }\Omega_{a, b, R}\setminus\{x'=0\}.
	\end{equation}
\end{lem}
\begin{proof}
	We prove the lemma when $a<b<0$. The proof when $0<a<b$ is similar. 
	
	Let $a', b'$ be arbitrary numbers satisfying $a<a'<b'<b$. For any fixed $\bar{x}\in \Omega_{a', b', R/2} \setminus\{x'=0\}$, let $\bar{\rho}:=\min\{|\bar{x}'|/3, (R-|\bar{x}'|)/3, (b-b')/3, (a'-a)/3\}$, so $B_{\bar{\rho}}(\bar{x})\subset \Omega_{a, b, R}$. Set $\bar{q}(y)=\bar{\rho}^{2-\delta}q(\bar{x}+\bar{\rho}y)$ and $\bar{F}(y)=\bar{\rho}^{3-\delta}F(\bar{x}+\bar{\rho}y)$ for $y\in B_2$. By (\ref{eq:lem:2_2:1}), we have $\sup_{B_2}(|\bar{q}|+\sum_{j=0}^{2}|\nabla^j\bar{F}|)=o(1)$ as $\bar{\rho}\to 0$. Moreover, 
	\[
		-\Delta \bar{q}(y)=\dive \bar{F},
		\quad y\in B_2. 
	\] 
	By elliptic theories, $\sup_{B_1}(|\nabla \bar{q}|+|\nabla^2 \bar{q}|)=o(1)$ as $\bar{\rho}\to 0$. Therefore, 
	\begin{equation}\label{eq:lem:2_2:pf:1}
		\sum_{j=1}^{2}|\nabla^j q(x)| |x'|^{j+2-\delta}=o(1), 
		\textrm{ as }x'\to 0, \textrm{ uniformly for }x_3\in (a', b'). 
	\end{equation}
	Denote $\Delta' = \partial_1^2 + \partial_2^2$. For each fixed $x_3\in (a, b)$, we have
	\[
		\Delta' q(x',x_3)=-\dive F(x', x_3)-\partial_3^2q(x',x_3)=:g_0(x', x_3)+g_1(x', x_3), 
	\]
	where $g_0(x', x_3):=-\partial_1F_1(x', x_3)-\partial_2F_2(x', x_3)$ 
	and $g_1(x', x_3):=-\partial_3F_3(x', x_3)-\partial_3^2q(x',x_3)$. 
	
	We first study the existence and regularity of the solutions $q_0$ and $q_1$ of the Poisson equation 
	\begin{equation}\label{eq:Poisson}
	\left\{
	\begin{aligned}
		& \Delta' q_i(x', x_3)=g_i(x', x_3), \quad x'\in D_R, \\
		& q_i(\cdot, x_3)|_{\partial D_R}=0,
	\end{aligned}
	\right.
	\end{equation}
	for $i=0,1$, and then estimate the remaining part $q_2:=q-q_0-q_1$. 
	
	(1) Since $F=o(|x'|^{\delta-2})$, we have 
	 \[
	 \sup_{a<x_3<b}\|g_0(\cdot, x_3)\|_{W^{-1, s}(D_R)}=\sup_{a<x_3<b} \| F(\cdot, x_3) \|_{L^{s}(D_{R})}<\infty, \quad \forall 1< s<\frac{2}{2-\delta}. 
	 \]
	So for each $x_3\in (a, b)$, there exists a solution $q_0(\cdot, x_3)\in W^{1, s}(D_R)$ of (\ref{eq:Poisson}) for $i=0$. Moreover, we have $\sup_{a<x_3<b} \|q_0(\cdot, x_3) \|_{W^{1, s}(D_{R})}<\infty$.

	Since $F$ is $(-3)$-homogeneous, $F_3(x)=|x_3|^{-3}F_3(-x_1/x_3, -x_2/x_3, -1)$. Without loss of generality, assume $a'<-1<b'$. 
	Then by (\ref{eq:lem:2_2:1}), we have
	\[
		|\partial_3F_3(x', x_3)|
		\le C(|F_3(-\frac{x}{x_3})|+ |x'| |\nabla F_3(-\frac{x}{x_3})|)=o\left(\frac{1}{|x'|^{2-\delta}}\right) 
	\]
	as $x'\to 0$, uniformly for $x_3\in (a, b)$, 
	where $C$ is some constant depending only on $a, b, R$. 
Moreover, since $q$ is $(-2)$-homogeneous, we have $q(x)=|x_3|^{-2}q(-x_1/x_3, -x_2/x_3, -1)$ in $\Omega_{a, b, R}\setminus\{x'=0\}$. 
	Then by (\ref{eq:lem:2_2:1}) and (\ref{eq:lem:2_2:pf:1}), we have
	\[
		|\partial_3^2q(x', x_3)|
		\le C(|q(-\frac{x}{x_3})|+|x'| |\nabla q(-\frac{x}{x_3})|+|x'|^2|\nabla^2 q(-\frac{x}{x_3})|)=o\left(\frac{1}{|x'|^{2-\delta}}\right)
			\]
	as $x'\to 0$, uniformly for $x_3\in (a', b')$, where $C$ is some constant depending only on $a, b, a', b', R$. 
	Thus we have $g_1(x', x_3)=o\left(|x'|^{\delta-2}\right)$ uniformly for $x_3\in (a', b')$ and 	 $\sup_{a'<x_3<b'}\|g_1(\cdot, x_3)\|_{L^s(D_R)}<\infty$ for any $1< s<\frac{2}{2-\delta}$. 
	In particular, since $a', b'$ are arbitrary numbers in $(a,b)$, we have $\|g_1(\cdot, x_3)\|_{L^s(D_R)}<\infty$ for any $x_3\in (a, b)$. Thus for each $x_3\in (a, b)$, there exists a solution $q_1(\cdot,x_3)\in W^{2,s}(D_R)$ of (\ref{eq:Poisson}) for $i=1$. Moreover, we have $\sup_{a'<x_3<b'} \| q_1(\cdot, x_3) \|_{W^{2,s}(D_{R})} < \infty$ for any $a<a'<b'<b$ and $1< s<\frac{2}{2-\delta}$, since $\sup_{a'<x_3<b'}\|g_1(\cdot, x_3)\|_{L^s(D_R)}<\infty$. So $q_0, q_1$ are well-defined in $\Omega_{a, b, R}$ and 
	\begin{equation}\label{eq:lem:2_2:pf:2}
		q_0, q_1 \in L^{\infty}_{loc}((a, b), W^{1, s}(D_R)), \quad \forall 1<s<\frac{2}{2-\delta}.
	\end{equation} 
	Since $g_0, g_1\in C^{\infty}(\bar{\Omega}_{a, b, R}\setminus\{x'=0\})$, we also have 
	$q_0, q_1\in C^{\infty}(\bar{\Omega}_{a, b, R}\setminus\{x'=0\})$. 	

	(2) Let $q_2(x',x_3):=q(x',x_3)-(q_0(x', x_3)+q_1(x', x_3))$, then $\Delta' q_2(x', x_3)=0$ in $\Omega_{a, b, R}\setminus\{x'=0\}$, and for any $a<a'<b'<b$, 
	\begin{equation}\label{eq:lem:2_2:pf:3}
		\sup\limits_{\substack{R/4<|x'|<R \\ a'<x_3<b'}}(|q_2|+|\nabla q_2|)<\infty. 
	\end{equation}
	We will show 
	\begin{equation}\label{eq:lem:2_2:pf:4}
		q_2(x)=\hat{h}(x)+a_0(x_3)\ln |x'|+a_1(x_3)\frac{x_1}{|x'|^2}+b_1(x_3)\frac{x_2}{|x'|^2}, \quad x\in \Omega_{a, b, R}\setminus\{x'=0\},
	\end{equation}
	for some $\hat{h}(x) \in L^{\infty}_{loc}((a, b), W^{1, s}(D_R))$ for any $1< s<\frac{2}{2-\delta}$ and $a_0(x_3), a_1(x_3), b_1(x_3)\in C(a, b)$. Then (\ref{eq:lem:2_2:2}) follows from (\ref{eq:lem:2_2:pf:2}) and (\ref{eq:lem:2_2:pf:4}) by setting $h=\hat{h}+q_0+q_1$. 
	
	For any $f(x')\in C^{\infty}(D_R\setminus\{0\})$, denote 
	\begin{equation}\label{eq:lem:2_2:pf:5}
		\hat{f}(x'):=f(x')-\left(\frac{1}{2\pi}\int_{\partial D_{R/2}}\frac{\partial f}{\partial \nu}\right)\ln |x'|.
	\end{equation}
	For each $x_3\in (a, b)$, let $\hat{q}_2(x', x_3)$ be defined as above. Since $\Delta' q_2=0$ in $D_R\setminus\{0\}$, we have $\Delta' \hat{q}_2=0$ in $D_R\setminus\{0\}$ and $\int_{\partial D_{R'}}\frac{\partial \hat{q}_2}{\partial \nu}=0$ for any $0<R'<R$ and $a<x_3<b$. Let $z=x_1+ix_2$, and
	\[
		w(z, x_3): 
		=\hat{q}_2+i \int_{(\frac{R}{2}, 0)}^z(-\partial_2\hat{q}_2 dx_1+\partial_1\hat{q}_2dx_2), 
	\]
	where the integral $\int_{(\frac{R}{2}, 0)}^{z}$ is independent of the path in $D_R\setminus\{0\}$. 
	Then $w$ is analytic in $z$ in $\{0<|z|< R\}$. 
	The Laurent series of $w$ in $z$ takes the form $w(z, x_3)=\sum_{m=-\infty}^{\infty} c_m(x_3) z^m$. 
	For any $0<R_1<R_2<R$, the series is uniformly convergent in $\{ R_1\le |z|\le R_2\}$. For 
	any $\rho>0$, we have 
	\begin{equation}\label{eq:lem:2_2:pf:6}
		c_m(x_3)=\frac{1}{2\pi i}\oint_{|z|=\rho} \frac{w(z, x_3)}{z^{m+1}}. 
	\end{equation}
	Now we show that $c_m(x_3)\equiv 0$ in $(a', b')$ for all $m\le -2$, 
	and $\sup_{a'<x_3<b'}|c_{-1}(x_3)|<\infty$. 
	
	\noindent\emph{Claim}: 
	$\sup_{a'<x_3<b'} \|w(\cdot, x_3)\|_{L^{s}(D_{R})}<\infty$ for any $1<s<\frac{2}{2-\delta}$.
	
	We will prove the claim later. Suppose the claim holds. Note that (\ref{eq:lem:2_2:pf:6}) holds for any $\rho>0$, we have
	\[
		c_{m}(x_3)=\frac{1}{2\pi i}\oint_{|z|=\rho} \frac{w(z, x_3)}{z^{m+1}}=\frac{1}{\pi i \rho}\int_{\rho/2}^{\rho}\oint_{|z|=t} \frac{w(z, x_3)}{z^{m+1}}dS(z)dt=\frac{1}{\pi i \rho}\int_{D_{\rho}\setminus D_{\rho/2}}\frac{w(z, x_3)}{z^{m+1}}.
	\]
	Now fix some $1<s<\frac{2}{2-\delta}$. By the Claim and H\"{o}lder's inequality, we have
	\[
		\sup_{a'<x_3<b'}|c_{m}(x_3)|\le \frac{1}{\pi\rho}\sup_{a'<x_3<b'}\|w\|_{L^s(D_{\rho}\setminus D_{\rho/2})}\|z^{-m-1}\|_{L^{\frac{s}{s-1}}(D_{\rho}\setminus D_{\rho/2})} 
		\le C\rho^{-m-\frac{2}{s}}
	\]
	for all $\rho>0$ and some $C$ independent of $\rho$. 
	Since $s>1$, we have $-m-2/s>-m-2$. By sending $\rho\to 0$, we have $\sup_{a'<x_3<b'}|c_m(x_3)|=0$ for $m\le -2$. As $a<a', b'<b$ are arbitrary, we have $c_m(x_3)\equiv 0$ for $x_3\in (a, b)$ and $m\le -2$. Then $w(z, x_3)=\sum_{m=-1}^{\infty}c_m(x_3)z^m$. By the definition of $w$ and (\ref{eq:lem:2_2:pf:3}), we have $\sup_{\substack{R/4<|x'|<R \\ a'<x_3<b'}}|w|<\infty$. Taking $\rho=R/2$ in (\ref{eq:lem:2_2:pf:6}), we have $\sup_{a'<x_3<b'}|c_m(x_3)|\le C(R/2)^{-m}$ for $m\ge -1$ for some $C$ depending only on $\sup_{\substack{R/4<|x'|<R \\ a'<x_3<b'}}|w|$. So $\sup_{a'<x_3<b'}|c_{-1}(x_3)|<\infty$. Moreover, for $|z|\le R/4$, 
	\begin{equation}\label{eq:lem:2_2:pf:7}
		| \sum_{m=0}^{\infty} c_m(x_3) z^m | \le C \sum_{m=0}^{\infty} \frac{1}{2^m}<\infty. 
	\end{equation}
	Set $\hat{h}(x', x_3)=\Real (\sum_{m=0}^{\infty}c_m(x_3)z^m)$, $a_0(x_3)=\frac{1}{2\pi}\int_{\partial D_{R/2}}\frac{\partial q_2}{\partial \nu}$, $a_1(x_3)=\Real c_{-1}(x_3)$ and $b_1(x_3)=\Imagine c_{-1}(x_3)$. By the fact that 
	\[
		\Real w(z, x_3)=\hat{q}_2=q_2(x', x_3)-(\frac{1}{2\pi}\int_{\partial D_{R/2}}\frac{\partial q_2}{\partial \nu})\ln |x'|, 
	\]
	we have for any $x_3\in(a,b)$ that
	\[
	\begin{split}
		q_2(x', x_3) & =\Real w(z, x_3)+(\frac{1}{2\pi}\int_{\partial D_{R/2}}\frac{\partial q_2}{\partial \nu})\ln |x'|\\
		& =\Real (\sum_{m=-1}^{\infty}c_m(x_3)z^m)+(\frac{1}{2\pi}\int_{\partial D_{R/2}}\frac{\partial q_2}{\partial \nu})\ln |x'|\\
		& =\hat{h}(x', x_3)+
		a_0(x_3)\ln |x'|+a_1(x_3)\frac{x_1}{|x'|^2}+b_1(x_3)\frac{x_2}{|x'|^2}. 
		\end{split}
	\]
	So (\ref{eq:lem:2_2:pf:4}) is proved. 
	
	Let $h:=\hat{h}+q_0+q_1$, then
	\[
		q=(q_0+q_1)+q_2=h(x', x_3)+a_0(x_3)\ln |x'|+a_1(x_3)\frac{x_1}{|x'|^2}+b_1(x_3)\frac{x_2}{|x'|^2}. 
	\]
	By (\ref{eq:lem:2_2:pf:2}) and (\ref{eq:lem:2_2:pf:7}), we have $h\in L^{\infty}_{loc}((a, b), W^{1, s}(D_R))$. By (\ref{eq:lem:2_2:pf:3}), we have $a_0\in L^{\infty}_{loc}(a, b)$. Since $\sup_{a'<x_3<b'}|c_{-1}(x_3)|\le C$, we have $a_1, b_1\in L^{\infty}_{loc}(a, b)$. Since $q\in C^{\infty}(\bar{\Omega}_{a, b, R}\setminus\{x'=0\})$, we have $q_1, q_2, \hat{q}_2\in C^{\infty} (\Omega_{a, b, R}\setminus\{x'=0\})$. Thus for each $z\ne 0$, $w(z, x_3)$ is continuous in $x_3\in (a, b)$ and $c_m$ is continuous in $x_3\in (a, b)$, and therefore $a_0(x_3), a_1(x_3)$, $b_1(x_3)\in C(a, b)$. 

	(3)	\noindent\emph{Proof of Claim}: 
	Recall that $q_2=q-(q_0+q_1)$, we have $\Real w=\hat{q}_2=\hat{q}-(\hat{q}_1+\hat{q}_0)$, where $\hat{q}(\cdot, x_3), \hat{q}_0(\cdot, x_3), \hat{q}_1(\cdot, x_3)$ are defined by (\ref{eq:lem:2_2:pf:5}) for each $x_3\in (a', b')$. By (\ref{eq:lem:2_2:1}) and (\ref{eq:lem:2_2:pf:2}), we have 
	\[
		\sup_{a'<x_3<b'}\|\hat{q}_2\|_{L^s(D_R)}
		<\infty, \quad \forall 1<s<\frac{2}{2-\delta}. 
	\]
	It remains to show 
	\begin{equation}\label{eq:lem:2_2:pf:8}
		\sup_{a'<x_3<b'} \|\Imagine w\|_{L^{s}(D_{R})}<\infty, \quad \forall 1<s<\frac{2}{2-\delta}.
	\end{equation} 
	Note $\Imagine w=\int_{(\frac{R}{2}, 0)}^z(-\partial_2\hat{q}_2 dx_1+\partial_1\hat{q}_2dx_2)$. For any $z\in D_R\setminus \{z_2=0\}$, denote $\bar{z}=\frac{R}{2}\frac{z}{|z|}$. Let $\Gamma_1$ be the counter-clockwise path from $(R/2, 0)$ to $\bar{z}$ along $\partial D_{R/2}$, and $\Gamma_2$ be the path from $\bar{z}$ to $z$ along the ray in the direction of $z$, and let $\Gamma=\Gamma_1\cup\Gamma_2$. For any $f=(f_1, f_2)\in C(D_R\setminus\{0\})$, define
	\begin{equation*}
		\mathcal{L}[f](z) := \int_{\Gamma}(f_1dx_1+f_2dx_2), \quad z\in D_R\setminus [0, R). 
	\end{equation*}
	Then the following facts hold.

	\noindent\emph{Fact 1}. If $f=(f_1, f_2)\in C(D_R\setminus\{0\})$ satisfies $|f_1(z)|+|f_2(z)|\le C_0|z|^{\lambda}$ for some $C_0>0$ and $\lambda\in\mathbb{R}\setminus\{-1\}$, then $|\mathcal{L}[f](z)|\le C(1+|z|^{\lambda+1})$ for some $C$ depending only on $R, \lambda$ and $C_0$. 
	
	To see this, let $(r, \theta)$ be the polar coordinates in $\mathbb{R}^2$, where $x_1=r\cos\theta$ and $x_2=r\sin\theta$. For any $z\in D_R\setminus \{z_2=0\}$, denote $z=(|z|, \theta_0)$, we have 
	\[
	\begin{split}
		|\mathcal{L}[f](z)| & =|\int_{\Gamma}(f_1\sin\theta+f_2\cos\theta)rd\theta+(f_1\cos\theta+f_2\sin\theta)dr|\\
		& \le \frac{R}{2}\int_{0}^{\theta_0}|f\big(\frac{R}{2}, \theta\big)|d\theta 
		+\int_{|z|}^{\frac{R}{2}}|f(r, \theta_0)|dr\\
		& 
		\le C+C|\int_{|z|}^{\frac{R}{2}}r^{\lambda}dr| 
		\le C+C|z|^{\lambda+1}.
	\end{split}
	\]
	
	\noindent\emph{Fact 2.} If $f=(f_1, f_2)\in C(D_R\setminus\{0\})\cap L^s(D_R)$ for some $s\ge 1$, then $\mathcal{L}[f]\in L^s(D_R)$, and 
	\[
		\|\mathcal{L}[f]\|_{L^s(D_R)}\le C(1+\|f\|_{L^{s}(D_R)}) 
	\]
	for some $C$ depending only on $s$ and $R$. 
	
	As in the proof of Fact 1, we have 
	\[
	\begin{split}
		|\mathcal{L}[f](z)| & \le \frac{R}{2}\int_{0}^{\theta_0}|f\big(\frac{R}{2}, \theta\big)|d\theta 
		+\int_{|z|}^{\frac{R}{2}}|f(r, \theta_0)|dr
		\le C+C\big|\int_{|z|}^{\frac{R}{2}}|f(r, \theta_0)|^sdr\big|^{\frac{1}{s}}. 
	\end{split}
	\]
	Taking the power $s$ of the above and integrating in $z$ over $D_R$, we have 
	\[
		\|\mathcal{L}[f](z)\|_{L^s(D_R)}\le C(1+\|f\|_{L^{s}(D_R)}). 
	\]
	So Fact 2 holds. 
	
	Since in the definition of $w$, the integral $\int_{(\frac{R}{2}, 0)}^{z}$ is independent of path, we take the path to be $\Gamma$ as defined above. Then 
	\[
		\Imagine w=\int_{\Gamma}(-\partial_2\hat{q}_2 dx_1+\partial_1\hat{q}_2dx_2)=\mathcal{L}[\nabla^{\perp}\hat{q}_2]=\mathcal{L}[\nabla^{\perp}\hat{q}]-\mathcal{L}[\nabla^{\perp}(\hat{q}_0+\hat{q}_1)].
	\]
	where $\nabla^{\perp}=(-\partial_2, \partial_1)$. By (\ref{eq:lem:2_2:pf:1}) and Fact 1, we have $\sup_{a'<x_3<b'}|z|^{2-\delta}|\mathcal{L}[\nabla^{\perp}\hat{q}(\cdot, x_3)]|<\infty$, and therefore $\sup_{a'<x_3<b'} \|\mathcal{L}[\nabla^{\perp}\hat{q}(\cdot, x_3)]\|_{L^{s}(D_{R})}<\infty$ for $1<s<\frac{2}{2-\delta}$. By (\ref{eq:lem:2_2:pf:2}) and Fact 2, we have $\sup_{a'<x_3<b'} \|\mathcal{L}[\nabla^{\perp}\hat{q}_0(\cdot, x_3)+\nabla^{\perp}\hat{q}_1(\cdot, x_3)]\|_{L^{s}(D_{R})}<\infty$ for $1<s<\frac{2}{2-\delta}$. Thus (\ref{eq:lem:2_2:pf:8}) holds and the Claim is proved. The lemma is proved. 
\end{proof}

\begin{lem}\label{lem:2_3}
	Let $R>0$ and $a<b$ satisfying $ab>0$. Suppose $(u,p)\in C^{\infty}(\Omega_{a, b, 3R}\setminus\{x'=0\})$ is a $(-1)$-homogeneous solution of the Navier-Stokes equations 
	\begin{equation}\label{eq:lem:2_3:1}
	\left\{
		\begin{array}{ll}
			 -\Delta u+u\cdot \nabla u + \nabla p=0, & \textrm{ in }\Omega_{a, b, 3R}\setminus\{x'=0\}\\
			\dive u=0, & \textrm{ in }\Omega_{a, b, 3R}\setminus\{x'=0\}
		\end{array}
	\right.
	\end{equation}
	satisfying 
	\begin{equation}\label{eq:lem:2_3:2}
		|u|=o(\ln |x'|), \quad \textrm{as }x'\to 0\textrm{ uniformly in } \Omega_{a, b, 3R}\setminus\{x'=0\}. 
	\end{equation}
	Then $p\in L^{\infty}_{loc}((a, b), W^{1, s}(D_R))$ for any $1< s<2$ and $a<a'<b'<b$, 	
	\begin{equation}\label{eq:lem:2_3:3}
		\| \nabla u \|_{L^2(\Omega_{a', b', R}\setminus\Omega_{a', b', \epsilon})}=o(\sqrt{| \ln \epsilon|}), \textrm{ as }\epsilon\to 0^+. 
	\end{equation}
\end{lem}
\begin{proof}
	Let $a', b', a'', b''$ be some arbitrary numbers such that $a<a''<a'<b'<b''<b$. For convenience, denote $\Omega_r=\Omega_{a', b', r}$ and $\tilde{\Omega}_r=\Omega_{a, b, r}$ for any $r>0$. Let $C$ denote a positive constant which may vary from line to line, depending only on $a, b, a', b', a'', b'', R$. For any $\epsilon>0$, let
	\[
		g_1(x') =
		\left\{
		\begin{array}{ll}
			0, & 0<|x'| \le \epsilon^2,\\
			\displaystyle -\frac{1}{\ln \epsilon}\ln \frac{|x'|}{\epsilon^2}, & \epsilon^2<|x'|<\epsilon,\\
			1, & \epsilon\le |x'|<R/4,\\
			\textrm{smooth between $0$ and $1$ and positive }, & R/4<|x'|<R/2,\\ 
			0,& |x'|>R/2. 
		\end{array}
		\right.
	\]
	Let $g_2(x_3)\ge 0$ be a cutoff function in $C^{\infty}_c(a, b)$ such that $g_2(x_3)=1$ for $x_3\in (a', b')$, and $g_2(x_3)=0$ for $x_3\le a''$ or $x_3\ge b''$. Let $g_{\epsilon}(x)=g_1(|x'|)g_2(x_3)$, then $g_{\epsilon}$ is compactly supported in $\tilde{\Omega}_{2R}$, satisfying 
	\begin{equation}\label{eq:lem:2_3:pf:1}
		\| \nabla g_{\epsilon} \|_{ L^{\infty} (\tilde{\Omega}_{2R} \setminus \tilde{\Omega}_{R/4} ) } + \| \partial_3 g_{\epsilon} \|_{ L^{\infty} (\tilde{\Omega}_{2R}) } \le C, \textrm{ and } \| \nabla' g_{\epsilon} \|_{ L^2 ( \tilde{\Omega}_{R/4} ) }
				\le \frac{C}{\sqrt{| \ln\epsilon|}},
	\end{equation}
	where $\nabla'=(\partial_1, \partial_2)$. 
	
	Taking divergence of the first equation in (\ref{eq:lem:2_3:1}), we have $\Delta p=-\dive (u\cdot \nabla u)$ in $\Omega_{3R}\setminus\{x'=0\}$. By (\ref{eq:lem:2_3:2}) and Lemma \ref{lem:2_1} with $\lambda=1$, we have $\sup_{a'<x_3<b'}|p|=o(| \ln |x'| |/|x'|)$ and $\sup_{a'<x_3<b'}\sum_{j=0}^{2}|\nabla^j (u\cdot \nabla u)| |x'|^{j+1}/| \ln|x'| |^2=o(1)$. Apply Lemma \ref{lem:2_2} for $q=p$, $F=u\cdot \nabla u$ and any $0<\delta<1$ there, we have 
	\begin{equation}\label{eq:lem:2_3:pf:2}
		p(x)=h(x)+a_0(x_3)\ln |x'|+a_1(x_3)\frac{x_1}{|x'|^2}+b_1(x_3)\frac{x_2}{|x'|^2}, \quad x\in \tilde{\Omega}_{3R}, 
	\end{equation}
	for some $a_0(x_3), a_1(x_3), b_1(x_3)\in C(a, b)$, and $h\in L^{\infty}_{loc}((a, b), W^{1, s}(D_{3R}))$ for any $1<s<2$. By Sobolev embedding, we have $h\in L^{\infty}_{loc}((a, b), L^{r}(D_{3R}))$ for any $1<r<\infty$. 
	
	We first prove (\ref{eq:lem:2_3:3}). Take the dot product of the first equation in (\ref{eq:lem:2_3:1}) with $g_{\epsilon}^2u$ and integrate on $\tilde{\Omega}_{2R}$, we have, using $\nabla \cdot u=0$, that 
	\[
	\begin{split}
		& 0=\int_{\tilde{\Omega}_{2R}}(-\Delta u+\nabla p+u\cdot \nabla u)\cdot (g^2_{\epsilon}u) dx\\
		&=\int_{\tilde{\Omega}_{2R}}\nabla u\cdot \nabla (g^2_{\epsilon}u)dx-2\int_{\tilde{\Omega}_{2R}}p g_{\epsilon} \nabla g_{\epsilon}\cdot udx-\int_{\tilde{\Omega}_{2R}}|u|^2g_{\epsilon}u\cdot \nabla g_{\epsilon}dx\\
		&=\int_{\tilde{\Omega}_{2R}}|\nabla (g_{\epsilon}u)|^2dx-\int_{\tilde{\Omega}_{2R}}|u|^2|\nabla g_{\epsilon}|^2 dx-2\int_{\tilde{\Omega}_{2R}}p g_{\epsilon} \nabla g_{\epsilon}\cdot udx-\int_{\tilde{\Omega}_{2R}}|u|^2g_{\epsilon}u\cdot \nabla g_{\epsilon}dx.
	\end{split}
	\]
	Note that $g_{\epsilon}=1$ in $\Omega_{R/4}\setminus\Omega_{\epsilon}$ and $g_{\epsilon}=0$ in $\tilde{\Omega}_{\epsilon^2}$. By this, the above, (\ref{eq:lem:2_3:2}), (\ref{eq:lem:2_3:pf:1}), and the fact that $p=h+O\left(1/|x'|\right)$ in $(a'', b'')\times D_{3R}$ (by (\ref{eq:lem:2_3:pf:2})) with some $h\in L^{\infty}_{loc}((a, b), L^{r}(D_{3R}))$ for any $1<r<\infty$, we have, as $\epsilon\to 0$, that
	\[
	\begin{split}
		& \int_{\Omega_{R/4}\setminus\Omega_{\epsilon}}|\nabla u|^2dx 
		\le \int_{\tilde{\Omega}_{2R}}|\nabla (g_{\epsilon}u)|^2dx\\
		& \le \int_{\tilde{\Omega}_{2R}}|u|^2|\nabla g_{\epsilon}|^2 dx+2\int_{\tilde{\Omega}_{2R}}|p g_{\epsilon} \nabla g_{\epsilon}u|dx+\int_{\tilde{\Omega}_{2R}}|u|^3|g_{\epsilon}| |\nabla g_{\epsilon}|\\
		& = O(1)+o(1)\int_{\tilde{\Omega}_{2R}\setminus\tilde{\Omega}_{\epsilon^2}}\left(| \ln |x'| |^2|\nabla g_{\epsilon}|^2 +|h| | \ln |x'|||\nabla g_{\epsilon}| \right.
		+\frac{| \ln |x'||}{|x'|} |\nabla g_{\epsilon}|+| \ln |x'| |^3|\nabla g_{\epsilon}|)\\
		& = O(1)+o(1)
		\left( \int_{ \tilde{\Omega}_{2R} \setminus \tilde{\Omega}_{\epsilon^2} } | \ln \epsilon |^2 |\nabla' g_{\epsilon} |^2 dx 
		+ \| \frac{| \ln |x'|| }{ |x'| } \|_{ L^2 ( \tilde{\Omega}_{2R} \setminus \tilde{\Omega}_{\epsilon^2} ) } \| \nabla' g_{\epsilon} \|_{ L^2 ( \tilde{\Omega}_{2R} \setminus \tilde{\Omega}_{\epsilon^2} ) } \right) \\
		& = o ( \ln \epsilon).
	\end{split}
	\]
	So (\ref{eq:lem:2_3:3}) is proved.
 
	Next, we prove $a_1(x_3)\equiv b_1(x_3)\equiv 0$ in (\ref{eq:lem:2_3:pf:2}) for $x\in \Omega_{R}$, and therefore $p=h+a_0(x_3)\ln |x'|\in L^{\infty}_{loc}((a, b), W^{1, s}(D_R))$ for any $1< s<2$. We first show that $b_1(x_3)\equiv 0$. Suppose $b_1(\bar{x}_3)\ne 0$ for some $\bar{x}_3\in (a, b)$. Without loss of generality, we assume that $b_1(\bar{x}_3)>0$. Choose $a', b', a'', b''$ such that $a<a''<a'<\bar{x}_3<b'<b''<b$ and $b_1(x_3)>b_1(\bar{x}_3)/2$ for $x_3\in (a'', b'')$. We take a cutoff function $g_2(x_3)$ as described earlier using these values of $a', b', a''$ and $b''$, and let $g_{\epsilon}(x)=g_1(|x'|)g_2(x_3)$. 
	
	By the first equation in (\ref{eq:lem:2_3:1}), we have 
	\[
		\Delta u_1-u\cdot\nabla u_1=\partial_1p=\partial_1h+a_0\frac{x_1}{|x'|^2}+a_1\frac{x_2^2-x_1^2}{|x'|^4}-2b_1\frac{x_1x_2}{|x'|^4}, 
	\]
	where $u_1$ is the first component of $u$. Multiplying the above by $g_{\epsilon}\frac{x_1x_2}{|x'|^2}$ and integrating on $\tilde{\Omega}_{2R}$, we have 
	\begin{equation}\label{eq:lem:2_3:pf:3}
		2\int_{\tilde{\Omega}_{2R}} b_1(x_3)g_{\epsilon}\frac{x_1^2x_2^2}{|x'|^6}= 
		\int_{\tilde{\Omega}_{2R}}(-\Delta u_1+u\cdot\nabla u_1+\partial_1h+a_0\frac{x_1}{|x'|^2}+a_1\frac{x_2^2-x_1^2}{|x'|^4})g_{\epsilon}\frac{x_1x_2}{|x'|^2}. 
	\end{equation}
	Since $g_{\epsilon}=g_{\epsilon}(|x'|, x_3)$, by the oddness in $x_2$ of the integrants, we have 
	\begin{equation*}
		\int_{\tilde{\Omega}_{2R}}a_0\frac{x_1}{|x'|^2}g_{\epsilon}\frac{x_1x_2}{|x'|^2}= \int_{\tilde{\Omega}_{2R}} a_1\frac{x_2^2-x_1^2}{|x'|^4}g_{\epsilon}\frac{x_1x_2}{|x'|^2}=0. 
	\end{equation*}
	By (\ref{eq:lem:2_3:3}) and (\ref{eq:lem:2_3:pf:1}), we have 
	\begin{equation*}
	\begin{split}
		& |\int_{\tilde{\Omega}_{2R}}\Delta u_1g_{\epsilon}\frac{x_1x_2}{|x'|^2}| =| \int_{\tilde{\Omega}_{2R}} \nabla u_1\nabla (g_{\epsilon}\frac{x_1x_2}{|x'|^2})|\\
		& \le C\|\nabla u_1\|_{L^2(\tilde{\Omega}_{2R}\setminus\tilde{\Omega}_{\epsilon^2})}(\|\nabla g_{\epsilon}\|_{L^2(\tilde{\Omega}_{2R}\setminus\tilde{\Omega}_{\epsilon^2})}+\|\frac{1}{|x'|}\|_{L^2(\tilde{\Omega}_{2R}\setminus\tilde{\Omega}_{\epsilon^2})})=o(\ln \epsilon).
	\end{split}
	\end{equation*}
	By (\ref{eq:lem:2_3:3}), the fact that $u=o(\ln |x'|)$ and the definition of $g_{\epsilon}$, we have 
	\begin{equation*}
		|\int_{\tilde{\Omega}_{2R}}u\cdot \nabla u_1g_{\epsilon}\frac{x_1x_2}{|x'|^2}| \le C\|\nabla u_1\|_{L^2(\tilde{\Omega}_{2R}\setminus\tilde{\Omega}_{\epsilon^2})}\|u\|_{L^2(\tilde{\Omega}_{2R}\setminus\tilde{\Omega}_{\epsilon^2})}=o(\sqrt{| \ln \epsilon|}).
	\end{equation*}
	By the definition of $g_{\epsilon}$, (\ref{eq:lem:2_3:pf:1}) and the fact that $h\in L^2_{loc}(\tilde{\Omega}_{2R})$, 
	we have 
	\begin{equation*}
	\begin{split}
		& |\int_{\tilde{\Omega}_{2R}}\partial_1hg_{\epsilon}\frac{x_1x_2}{|x'|^2}| =|\int_{\tilde{\Omega}_{2R}}h(\partial_1g_{\epsilon}\frac{x_1x_2}{|x'|^2}+g_{\epsilon}\partial_1 \frac{x_1x_2}{|x'|^2})|\\
		& \le \|h\|_{L^2(\Omega_R)}(\|\nabla' g_{\epsilon}\|_{L^2(\tilde{\Omega}_{2R})}+\|\frac{1}{|x'|}\|_{L^2(\tilde{\Omega}_{2R}\setminus\tilde{\Omega}_{\epsilon^2})})=O(\sqrt{| \ln\epsilon|}). 
		\end{split}
	\end{equation*}
	So the right hand side of (\ref{eq:lem:2_3:pf:3}) is $o(\ln \epsilon)$. On the other hand, by the definition of $g_{\epsilon}$, the integral on the left hand side of (\ref{eq:lem:2_3:pf:3}) satisfies 
	\begin{equation*}
		\int_{ \tilde{\Omega}_{2R} } b_1(x_3)g_{\epsilon} \frac{ x_1^2 x_2^2 }{ |x'|^6 } \ge \frac{b_1(\bar{x}_3)}{C} \int_{ ( \Omega_{R/4} \setminus \Omega_{\epsilon} ) \cap \{ 2 |x_1| \le |x'| \le 3 |x_1| \} } \frac{1}{ |x'|^2 } 
		\ge \frac{b_1(\bar{x}_3)}{C} | \ln \epsilon |. 
	\end{equation*}
	This contradicts to $b_1(\bar{x}_3)>0$. We have proved that $b_1(x_3)\equiv 0$. Similarly, we have $a_1(\bar{x}_3)\equiv 0$. The lemma follows from this and (\ref{eq:lem:2_3:pf:2}). 
\end{proof}

\begin{lem}\label{lem:2_4}
	Let $R>0$, $a<b$ satisfying $ab>0$, and $g\in C^{\infty}(\Omega_{a, b, R}\setminus\{x'=0\})$. Suppose $(v, q)$ is a $C^{\infty}$ solution of 
	\begin{equation}\label{eq:lem:2_4:1}
	\left\{
	\begin{array}{ll}
			-\Delta v+\nabla q=\dive g, & \textrm{ in }\Omega_{a, b, R}\setminus\{x'=0\}, \\
			\dive v=0, & \textrm{ in }\Omega_{a, b, R}\setminus\{x'=0\},
	\end{array}
	\right.
	\end{equation}
	satisfying $v\in L^2(\Omega_{a, b, R})$, $\nabla v, g\in L^1(\Omega_{a, b, R})$, and 
	\begin{equation}\label{eq:lem:2_4:2}
		\|\nabla v\|_{L^2(\Omega_{a, b, R}\setminus\Omega_{a, b, \epsilon})}+\|q\|_{L^2(\Omega_{a, b, R}\setminus\Omega_{a, b, \epsilon})} +\|g\|_{L^2(\Omega_{a, b, R}\setminus\Omega_{a, b, \epsilon})}
		=o( \sqrt{| \ln \epsilon|}), \textrm{ as }\epsilon\to 0^+. 
	\end{equation}
	Then for any $\varphi\in C_c^{\infty}(\Omega_{a, b, R})$ satisfying $\dive \varphi=0$, it holds that 
	\begin{equation}\label{eq:lem:2_4:3}
		\int_{\Omega_{a, b, R}}\nabla v\cdot \nabla \varphi 
		+g\cdot \nabla \varphi=0, 
	\end{equation}
	and $ \dive v=0$ in $\Omega_{a, b, R}$ 
	in distribution sense. 
\end{lem}
\begin{proof}
	For convenience, denote $\Omega=\Omega_{a, b, R}$ and $\Omega_{\epsilon}=\Omega_{a, b, \epsilon}$. Let $\varphi\in C_c^{\infty}(\Omega)$ satisfy $\dive \varphi=0$, we first prove (\ref{eq:lem:2_4:3}). 

	Let 
	\[
		h_{\epsilon}(x): =
		\left\{
		\begin{array}{ll}
			0, & 0<|x'| \le \epsilon,\\
			\displaystyle -\frac{2}{\ln \epsilon}\ln \frac{|x'|}{\epsilon}, & \epsilon<|x'|<\sqrt{\epsilon},\\
			1, & |x'| \ge \sqrt{\epsilon}.
		\end{array}
		\right.
	\]
	We have, for some constant $C$ independent of $\epsilon$, that
	\begin{equation}\label{eq:lem:2_4:pf:1}
		\|\nabla h_{\epsilon}\|_{L^2(\Omega)}\le \frac{C}{\sqrt{| \ln \epsilon|}}.
	\end{equation}
	Multiplying the first equation in (\ref{eq:lem:2_4:1}) by $h_{\epsilon}\varphi$ and integrating on $\Omega$, we have 
	\[
		\int_{\Omega}\nabla v\cdot \nabla (h_{\epsilon}\varphi)=\int_{\Omega} q\dive (h_{\epsilon}\varphi)-\int_{\Omega}g\cdot \nabla(h_{\epsilon}\varphi).
	\]
	So 
	\begin{equation}\label{eq:lem:2_4:pf:2}
	\begin{split}
		& \int_{\Omega}\nabla v\cdot \nabla \varphi
		+ g \cdot \nabla \varphi\\
		& =\int_{\Omega} \left(q\varphi\cdot \nabla h_{\epsilon}-\nabla v\cdot (\nabla h_{\epsilon}\otimes\varphi)
		- g\cdot (\nabla h_{\epsilon}\otimes\varphi) \right) 
		+\int_{\Omega}(\nabla v\cdot \nabla \varphi
		+g \cdot\nabla\varphi)(1-h_{\epsilon})\\
		& :=I_1+I_2. 
	\end{split}
	\end{equation}
	By (\ref{eq:lem:2_4:2}) and (\ref{eq:lem:2_4:pf:1}), we have 
	\begin{equation}\label{eq:lem:2_4:pf:3}
		|I_1|\le (\|\nabla v\|_{L^2(\Omega_{\sqrt{\epsilon}}\setminus\Omega_{\epsilon})}+\|q\|_{L^2(\Omega_{\sqrt{\epsilon}}\setminus\Omega_{\epsilon})}+\|g\|_{L^2(\Omega_{\sqrt{\epsilon}}\setminus\Omega_{\epsilon})})\|\nabla h_{\epsilon}\|_{L^2(\Omega_{\sqrt{\epsilon}}\setminus\Omega_{\epsilon})}\|\varphi\|_{L^{\infty}(\Omega)}\to 0, 
	\end{equation}
	as $\epsilon\to 0$. Since $\nabla v, g\in L^1(\Omega_{a, b, R})$, we also have 
	\begin{equation}\label{eq:lem:2_4:pf:4}
		|I_2|\le \|\nabla \varphi\|_{L^{\infty}(\Omega)}(\|\nabla v\|_{L^1(\Omega_{\sqrt{\epsilon}})} 
		+\|g\|_{L^1(\Omega_{\sqrt{\epsilon}})})\to 0, \quad \textrm{ as }\epsilon\to 0.
	\end{equation}
	By (\ref{eq:lem:2_4:pf:2}), (\ref{eq:lem:2_4:pf:3}) and (\ref{eq:lem:2_4:pf:4}), we have (\ref{eq:lem:2_4:3}). 
	
	Now we prove that $\dive v=0$ in distribution sense. For any $\psi\in C_c^{\infty}(\Omega)$, multiplying the second equation (\ref{eq:lem:2_4:1}) by $h_{\epsilon}\psi$ and integrating on $\Omega$, we have 
	\[
		\int_{\Omega}v\cdot \nabla (h_{\epsilon}\psi)=0.
	\]
	Using the above, the fact that $v\in L^2(\Omega)$ and (\ref{eq:lem:2_4:pf:1}), we have 
	\[
	\begin{split}
		& |\int_{\Omega}v\cdot \nabla \psi|\le |\int_{\Omega}v\cdot \nabla (h_{\epsilon}\psi)|+|\int_{\Omega}v\cdot \nabla (\psi (h_{\epsilon}-1))|\\
		& \le \|\psi\|_{L^{\infty}(\Omega)}\|v\|_{L^2(\Omega_{\sqrt{\epsilon}})}\|\nabla h_{\epsilon}\|_{L^2(\Omega_{\sqrt{\epsilon}}\setminus\Omega_{\epsilon})}+\|\nabla \psi\|_{L^{\infty}(\Omega)}\|v\|_{L^{1}(\Omega_{\sqrt{\epsilon}})}\to 0, \textrm{ as }\epsilon\to 0.
	\end{split}
	\]
	So $\dive v=0$ in distribution sense. 
	The lemma is proved. 
\end{proof}

For any domain $\Omega\subset \mathbb{R}^n$, $s\ge 1$ and $f\in W^{-1, s}(\Omega)$, 
we say that $(v, q)\in W^{1,s}(\Omega)\times L^s(\Omega)$ is a $s$-weak solution to the Stokes system 
\[
\left\{
\begin{array}{ll}
	-\Delta v+\nabla q=f, & \textrm{ in }\Omega\\
	\dive v=0, & \textrm{ in }\Omega
\end{array}
\right.
\]
if for any $\varphi\in C_c^{\infty}(\Omega)$ satisfying $\dive \varphi=0$, it holds that 
\begin{equation}
	\int_{\Omega} \nabla v\cdot \nabla \varphi= \langle f, \varphi \rangle, 
\end{equation}
and $\dive v=0$ in distribution sense. Here $\langle \cdot, \cdot\rangle$ denotes the pairing between $W^{-1, s}(\Omega)\equiv (W^{1, s}_0(\Omega))'$ and $W^{1, s}_0(\Omega)$. 

\medskip

\noindent\emph{Proof of Theorem \ref{thm_main}}: 

For convenience, denote $\Omega_r=\Omega_{-2, -1/2, r}$ for any $r>0$. Let $\Omega:=\Omega_{\delta/16}$. 

By (\ref{eq:thm:main}), $u=o(\ln |x'|)$. Applying Lemma \ref{lem:2_1} with $\lambda=1$, $a=-3, b=-1/4, a'=-2, b'=-1/2$ and $R=\delta/16$, we have $\nabla u=o(| \ln |x'||/|x'|)$ and $p=o(| \ln |x'||/|x'|)$ in $\Omega$. It follows that $|u|^2\in L^s(\Omega)$, $u\in W^{1, s}(\Omega)$ and $p\in L^s(\Omega)$ for $1<s<2$. 

\medskip

\noindent \emph{Claim}: For any $1<s<2$, $(u, p)$ is a $s$-weak solution of 
\begin{equation}\label{eq:thm:pf:1}
\left\{
\begin{array}{ll}
	-\Delta u+\nabla p=f:=-u\cdot \nabla u, 
	& \textrm{ in }\Omega,\\ 
	\textrm{div } u=0, & \textrm{ in }\Omega.
\end{array}
\right.
\end{equation}
\noindent\emph{Proof of the claim}: 
We first show that for any $\varphi\in C_c^{\infty}(\Omega)$ satisfying $\dive \varphi=0$, 
\begin{equation}\label{eq:thm:pf:2}
	\int_{\Omega}\nabla u\cdot \nabla \varphi= \langle f, \varphi \rangle = \int_{\Omega}(u\otimes u)\cdot \nabla \varphi.
\end{equation}
Since $u=o(\ln |x'|)$ and $\nabla u=o(| \ln |x'||/|x'|)$, we have $u\in L^2(\Omega)$ and $\nabla u, u\otimes u\in L^1(\Omega)$. For any $\epsilon>0$, by Lemma \ref{lem:2_3} with $a=-3, b=-1/4, a'=-2, b'=-1/2$ and $R=\delta/16$, we have $\|\nabla u\|_{L^2(\Omega \setminus\Omega_{\epsilon})}=o(\sqrt{| \ln \epsilon|})$, and $p\in L^{\infty}_{loc}((-3, -1/4), W^{1, s}(D_{\delta/16}))$. By Sobolev embedding, $p\in L^{\infty}_{loc}((-3, -1/4), L^r(D_{\delta/16}))$ for any $1< r<\infty$, and therefore $\|p\|_{L^2(\Omega \setminus\Omega_{\epsilon})}=O(1)$. Thus 
\[
	\|\nabla u\|_{L^2(\Omega \setminus\Omega_{\epsilon})}+\|p\|_{L^2(\Omega \setminus\Omega_{\epsilon})} +\|u\otimes u\|_{L^2(\Omega \setminus\Omega_{\epsilon})}
	= o ( \sqrt{| \ln \epsilon|}), \textrm{ as }\epsilon\to 0^+. 
\]
Applying Lemma \ref{lem:2_4} with $a=-2, b=-1/2, R=\delta/16$, $v=u$, $q=p$ and $g=-u\otimes u$, we have that (\ref{eq:thm:pf:2}) holds and $\dive u=0$ in distribution sense, and therefore $(u, p)$ is a $s$-weak solution to (\ref{eq:thm:pf:1}) for any $1<s<2$. The claim is proved. 

Next, note $f=-u\cdot \nabla u\in L^s(\Omega)$ for any $1<s<2$. By interior estimates of Stokes equations (see, e.g. Theorem IV.4.1 in \cite{Galdi}), we have $(u, p)\in W^{2, s}_{loc}(\Omega)\times L^s_{loc}(\Omega)$. By Sobolev embedding, we have $u\in W^{1, s'}_{loc}(\Omega)\cap L^{\infty}_{loc}(\Omega)$ for all $1<s'<6$. Then we have $f=-u\cdot \nabla u\in W^{1, s}_{loc}(\Omega)$ for any $1<s<2$. By bootstrap argument, we have $(u, p)\in W^{m, s}_{loc}(\Omega)\times L^s_{loc}(\Omega)$ for any integer $m\ge 0$. Thus $(u, p)\in C^{\infty}(\Omega)$. The theorem is proved. 
\qed

\section{Discussion on isolated singularity behavior}\label{sec:3}

In this section, we first make some discussion on the isolated singularity behavior of $(-1)$-homogeneous solutions of (\ref{eq:NS}), then discuss some known special solutions of (\ref{eq:NS}) with isolated singularities on $\mathbb{S}^2$. 

\subsection{Types of singularities}

As mentioned in Section \ref{sec:intro}, our efforts start from studying $(-1)$-homogeneous solutions $u$ of (\ref{eq:NS}) in $C^2(\mathbb{S}^2\setminus\{P_1, \cdots, P_m\})$ with finitely many isolated singularities $P_1, \cdots, P_m$ on $\mathbb{S}^2$, where $m$ is any positive integer. In particular, we would like to investigate the asymptotic behavior near the singularities and the classification of solutions satisfying $u=O(1/\min_i {\rm dist}(x, P_i)^k)$ for some positive integer $k$. Consider a local $(-1)$-homogeneous solution $u$ in a small neighborhood of a singular point. Without loss of generality, assume the singularity is at $S$, i.e. $u\in C^2(B_{\delta}(S)\cap \mathbb{S}^2\setminus\{S\})$, and $u=O(1/|x'|^k)$. The first step is to understand the behavior of $u$ in $B_{\delta}(S)\cap \mathbb{S}^2\setminus\{S\}$. In \cite{LLY1} and \cite{LLY2}, the asymptotic expansions of all local $(-1)$-homogeneous axisymmetric solutions of (\ref{eq:NS}) in $B_{\delta}(S)\cap \mathbb{S}^2\setminus\{S\}$ were established. In particular, the following theorem can be derived from there. 
 
\begin{theorem}[\cite{LLY1,LLY2}] \label{thm:expansion}
	Let $\delta>0$, and $u\in C^2(\mathbb{S}^2 \cap B_\delta(S)\setminus\{S\})$ be a $(-1)$-homogeneous axisymmetric solution of (1). Denote $x'=(x_1,x_2)$. Then $\tau:= \lim_{x\in\mathbb{S}^2,x\to S} |x'| u_\theta$ exists and is finite, and $u = O(1/\big| |x'| \ln |x'| \big|^2)$. Moreover, 
		
	{\rm (i)} If $\tau \ge 3$, then $|x'| u_\phi$ must be a constant, and $|x'| u_\theta$ and $u_r$ must be real analytic functions in $1+\cos\theta$ near $S$ on $\mathbb{S}^2$. 
	
	{\rm (ii)} If $2 < \tau < 3$, then either $|x'|u_\phi \equiv $constant, or $\lim_{x\in\mathbb{S}^2,x\to S} |x'|^{\tau-1} u_\phi$ exists and is finite and not zero. Moreover, $\lim_{x\in\mathbb{S}^2,x\to S} |x'|^{2\tau-4} u_r$ exists and is finite. 
	
	{\rm (iii)} If $\tau=2$, then $\eta:= \lim_{x\in \mathbb{S}^2, x\to S} (|x'|u_\theta-2)\ln |x'|$ exists and is 0 or 2. 
	
	- When $\eta=0$, then $\lim_{x\in\mathbb{S}^2,x\to S} |x'|^{\epsilon} u_r = 0$ for any $\epsilon>0$. Either $|x'| u_\phi$ is a constant, or $\lim_{x\in\mathbb{S}^2,x\to S} |x'| u_\phi$ exists and is finite and not zero. 
	
	- When $\eta=2$, then $\lim_{x\in\mathbb{S}^2,x\to S} |x'|^2 \big| \ln |x'| \big|^2 u_r = -2$, and $\lim_{x\in\mathbb{S}^2,x\to S}|x'|u_\phi$ exists and is finite. 
	
	{\rm (iv)} If $\tau < 2$ and $\tau \not= 0$, then $\lim_{x\in\mathbb{S}^2,x\to S} |x'|u_\phi$ and $\lim_{x\in\mathbb{S}^2,x\to S} |x'|^{\max\{\tau,0\}} u_r$ both exist and are finite. 
	
	{\rm (v)} If $\tau=0$, then $\sigma:=\lim_{x\in\mathbb{S}^2,x\to S} |x'|u_\phi$ and 
	$\lim_{x\in\mathbb{S}^2,x\to S}u_r/\ln |x'|$ both exist and are finite. 
	 Moreover, $\tilde{u}:=u-\sigma/|x'|e_{\phi}$ is also a solution of (\ref{eq:NS}), and 
	 $\lim_{x\in\mathbb{S}^2,x\to S} |\tilde{u}|/\big| \ln |x'| \big|$ exists and is finite. 
\end{theorem}

\medskip

Similar results in Theorem \ref{thm:expansion} hold for solutions $u\in C^2(\mathbb{S}^2 \cap B_\delta(N)\setminus\{N\})$. Theorem \ref{thm:expansion} can be concluded from \cite{LLY1} and \cite{LLY2}, but some of the statements are not explicitly listed there. For the sake of completeness, here we briefly describe how Theorem \ref{thm:expansion} is concluded from the results in \cite{LLY1} and \cite{LLY2}. 

\begin{proof}
	Let $y:=\cos\theta$ and $U:= u\cdot r \sin\theta$. Since $u$ is $(-1)$-homogeneous and axisymmetric, we know that $U$ depends only on $y\in (-1, 1)$. Note $y=1$ and $-1$ correspond to the north and south pole $N$ and $S$ of $\mathbb{S}^2$ respectively, while $-1<y<1$ corresponds to $\mathbb{S}^2\setminus\{S,N\}$. The divergence free condition in (\ref{eq:NS}) is equivalent to $ru_r=dU_{\theta}/dy$, and (\ref{eq:NS}) is reduced to the following system of $U_{\theta}$ and $U_{\phi}$, 
	\begin{equation}\label{eq:NS:0}
		\left\{
		\begin{aligned}
			& (1-y^2) \frac{d}{dy}U_\theta + 2y U_\theta +\frac{1}{2} U_\theta^2 + \int_{y_0}^y \int_{y_0}^l \int_{y_0}^t \frac{2 U_\phi(s) \frac{d}{ds}U_\phi(s)}{1-s^2} ds dt dl = b_1 y^2 +b_2 y + b_3, \\
			& (1-y^2) \frac{d^2}{dy^2}U_\phi + U_\theta \frac{d}{dy}U_\phi = 0, 
		\end{aligned}
		\right. 
	\end{equation}
	for any $y_0\in (-1, 1)$ and some constants $b_1, b_2, b_3\in\mathbb{R}$. 
	
	By Theorem 1.3 in \cite{LLY1}, we have that $\tau:=\lim_{x\in\mathbb{S}^2, x\to S}|x'|u_{\theta}$ exists and is finite. 
	
	(i) By Theorem 1.4 in \cite{LLY1} we have that $|x'|u_{\phi}$ must be a constant when $\tau \ge 3$. So (\ref{eq:NS:0}) is reduced to 
	\begin{equation}\label{eq:NS:1}
		(1-y^2)\frac{d}{dy}U_{\theta}(y)+2yU_{\theta}(y)+\frac{1}{2}U^2_{\theta}(y)=c_1(1-y)+c_2(1+y)+c_3(1-y^2), 
	\end{equation}
	for some real constants $c_1, c_2, c_3$. Then by Theorem 1.1, Theorem 1.2 and Lemma 2.3 in \cite{LLY2}, we have that $|x'|u_{\theta}=U_{\theta}$ is a real analytic function in $1+\cos\theta$ near $S$ on $\mathbb{S}^2$. By the divergence free condition, $u_r=\frac{d}{dy}U_{\theta}(y)$ is also a real analytic function in $1+\cos\theta$ near $S$ on $\mathbb{S}^2$.
	
	(ii) When $2<\tau<3$, by Theorem 1.3 and 1.4 in \cite{LLY1}, we know for $x\in \mathbb{S}^2$ that 
	\begin{align}
		U_\theta = & \tau + a_1 |x'|^{6-2\tau} + a_2 |x'|^2 + O(|x'|^{4(3-\tau)-\epsilon}), \label{eq:Uth:2} \\
		U_\phi = & d_0 + d_1 \big( |x'|^{2-\tau} + d_2 |x'|^{8-3\tau} + d_3 |x'|^{4-\tau} + O(|x'|^{14-5\tau-\epsilon}) \big). \label{eq:Uph:2}
	\end{align}
	Note $U_{\phi}=u_{\phi}r\sin\theta=|x'|u_{\phi}$ and $\tau>2$. If $d_1=0$, then $|x'|u_\phi \equiv d_0$ is a constant. If $d_1\ne 0$, then $\lim_{x\in\mathbb{S}^2,x\to S} |x'|^{\tau-1} u_\phi=d_1$, which is finite and nonzero. 		
	
	The behavior of $u_r$ on $\mathbb{S}^2$ is obtained by using the fact $u_r=dU_{\theta}/dy$ on $\bS^2$ and the above behavior of $U_{\theta}$. Write the first equation in (\ref{eq:NS:0}) as 
	\begin{equation}\label{eq:NS:2}
		(1-y^2)u_r = - 2 y U_\theta - \frac{1}{2} U_\theta^2 + \int_{y_0}^{y} \frac{U_{\phi}^2 (s-y) (1-sy) }{ (1-s^2)^2 } ds + b_1 y^2 + b_2 y + b_3.
	\end{equation}
	Note that 
	\[
		(s-y)(1-sy) = - y (1+s)^2 + (1+y)^2 s. 
	\]
	Then the expansion of $u_r$ is obtained by substituting the expansions in (\ref{eq:Uth:2}) and (\ref{eq:Uph:2}) into (\ref{eq:NS:2}). 
	
	(iii) When $\tau=2$, by Theorem 1.3 and Theorem 1.4 in \cite{LLY1}, 
	\[
		U_\theta = 2 + \frac{\eta}{\ln|x'|} + O\big( | \ln |x'||^{-2+\epsilon} \big), 
	\]
	where $
		\eta:= \lim_{x\in \mathbb{S}^2, x\to S} (|x'|u_\theta-2) \ln |x'|=0\textrm{ or 2}$. 
	If $\eta=0$, then \[
		U_\phi = d_1 \ln |x'| + d_2 + d_1 O(|x'|^{1-\epsilon}). 
	\]
	If $d_1=0$, then $ |x'| u_\phi=U_\phi$ is a constant. If $d_1\not=0$, then $\lim_{x\in\mathbb{S}^2,x\to S} |x'| u_\phi / \ln |x'|=d_1$, which is finite and not zero. 

	If $\eta=2$, then 
	\[
		U_\phi = U_\phi(-1) + \frac{d_3}{\ln |x'|} + O\big( (\ln |x'|)^{-2+\epsilon} \big).
	\]
	So $U_\phi$ is bounded and $\lim_{x\in\mathbb{S}^2,x\to S}|x'|u_\phi$ exists and is finite. 
	
	The expansion of $u_r$ is obtained by substituting the expansions of $U_\theta$ and $U_\phi$ into (\ref{eq:NS:2}). It can be proved that: If $\eta=0$, then $\lim_{x\in\mathbb{S}^2,x\to S} |x'|^{\epsilon} u_r = 0$ for any $\epsilon>0$. If $\eta=2$, then $\lim_{x\in\mathbb{S}^2,x\to S} |x'|^2 \big| \ln |x'| \big|^2 u_r = -2$. 
	
	(iv) When $\tau<2$ and $\tau\not=0$, by Theorem 1.3 and Theorem 1.4 in \cite{LLY1}, we know that 
	\begin{align}
		U_\theta = & \tau + a_1 |x'|^{2-\tau} + a_2 |x'|^2 + O(|x'|^{4-2\tau-\epsilon}) + O(|x'|^{2-\epsilon}), \label{eq:Uth:4} \\
		U_\phi = & U_\phi(-1) + d_1 |x'|^{2-\tau} + d_2 |x'|^{4-2\tau} + d_3 |x'|^{4-\tau} + O(|x'|^{6-\tau-\epsilon}) + O(|x'|^{6-3\tau-\epsilon}). \label{eq:Uph:4}
	\end{align}
	It is easy to see that $\lim_{x\in\mathbb{S}^2,x\to S} |x'|u_\phi$ exists and is finite. The asymptotic behavior of $u_r$ can be obtained by substituting the expansions (\ref{eq:Uth:4}) and (\ref{eq:Uph:4}) into (\ref{eq:NS:2}). In the present case, $\lim_{x\in\mathbb{S}^2,x\to S} |x'|^{\max\{\tau,0\}} u_r$ exists and is finite.
	
	(v) When $\tau=0$, by Theorem 1.3 and Theorem 1.4 in \cite{LLY1}, we know that 
	\begin{align}
		U_\theta = & a_1 |x'|^2 \ln |x'| + a_2 |x'|^2 + O\big( |x'|^{4-\epsilon} \big), \label{eq:Uth:5} \\
		U_\phi = & U_\phi(-1) + d_1 |x'|^{2} + d_2 |x'|^{4} \ln |x'| + d_3 |x'|^4 + O\big( |x'|^{6-\epsilon} \big). \label{eq:Uph:5}
 	\end{align}
	Then $\lim_{x\in\mathbb{S}^2,x\to S} |x'|u_\phi$ exists and is finite. The asymptotic behavior of $u_r$ can be obtained by substituting the expansions (\ref{eq:Uth:5}) and (\ref{eq:Uph:5}) into (\ref{eq:NS:2}). Then we have
	\[
		u_r = c_1 \ln |x'| + O(1), 
	\]
	for some nonzero constant $c_1$, and $\lim_{x\in\mathbb{S}^2,x\to S} u_r/ \ln |x'|$ exists and is finite. Since for any solution $(U_{\theta}, U_{\phi})$, $(U_{\theta}, U_{\phi}-C)$ is also a solution of (\ref{eq:NS:0}) for any constant $C$. By taking $C=U_{\phi}(-1)$, we have $\tilde{u}=u-U_{\phi}(-1)/|x'|e_{\phi}$ is also a solution of (\ref{eq:NS}) and satisfies that $\lim_{x\in\mathbb{S}^2,x\to S} |\tilde{u}|/\big| \ln |x'| \big|$ exists and is finite by combining the above with (\ref{eq:Uth:5}) and (\ref{eq:Uph:5}). 	
	
	In the end, combining all above argument, we have $u=O( 1 / \big| |x'| \ln |x'| \big|^2) $ for all $(-1)$-homogeneous axisymmetric solutions, which completes the proof. 
\end{proof}

In view of Theorem \ref{thm:expansion}, 
all $(-1)$-homogeneous axisymmetric solutions of (\ref{eq:NS}) in $C^{\infty}(\mathbb{S}^2\setminus\{S, N\})$ are of the following three mutually exclusive types:
\begin{enumerate}
	\item[]Type 1. Landau solutions, satisfying $\sup_{|x|=1}|u(x)|<\infty$; 
	\item[]Type 2. Solutions satisfying $0<\limsup_{|x|=1,x'\to 0}|u(x)|/| \ln |x'||<\infty$; 
	\item[]Type 3. Solutions satisfying $\limsup_{|x|=1,x'\to 0}|x'||u(x)|>0$. 
\end{enumerate}
This classification is equivalent to the one in \cite{LY7}, which is given as follows: 
\begin{enumerate}
	\item[]Type 1'. Landau solutions, satisfying $\sup_{|x|=1}|\nabla u(x)|<\infty$; 
	\item[]Type 2'. Solutions satisfying 
	$0<\limsup_{|x|=1,x'\to 0}|x'||\nabla u(x)|<\infty$; 
	\item[]Type 3'. Solutions satisfying 
	$\limsup_{|x|=1,x'\to 0}|x'|^2|\nabla u(x)|>0$.
\end{enumerate}
Below we briefly explain why these two classifications are equivalent. 

To see Type 1 and Type 1' solutions are the same, we claim that a $(-1)$-homogeneous solution $u$ of (\ref{eq:NS}) in $C^{\infty}(\mathbb{S}^2\setminus\{S, N\})$ is a Landau solution if and only if $\sup_{|x|=1}|u(x)|<\infty$ if and only if $\sup_{|x|=1}|\nabla u(x)|<\infty$. Note Landau solutions are smooth on $\mathbb{S}^2$, thus satisfy $\sup_{|x|=1}(|u(x)|+|\nabla u(x)|)<\infty$. On the other hand, $\sup_{|x|=1}|\nabla u(x)|<\infty$ implies $\sup_{|x|=1}|u(x)|<\infty$. Moreover, if $\sup_{|x|=1}|u(x)|<\infty$, then $u=o(\ln |x'|)$ on $\mathbb{S}^2$ as $x\to N, S$. Then by Theorem \ref{thm_main}, $u$ is in $C^{\infty}(\mathbb{S}^2)$ and thus must be a Landau solution in view of Theorem \ref{thm:Sverak}.

By integrating the conditions in $|\nabla u|$ for Type 2' and Type 3' on $\mathbb{S}^2$ near $S$ and $N$, we see Type 2' solutions must be of Type 2, and Type 3' solutions must be of Type 3.

Now suppose $u$ is a $(-1)$-homogeneous axisymmetric solution of Type 2, we show it must be of Type 2'. Since $u$ is $(-1)$-homogeneous and axisymmetric, we have 
\begin{equation}\label{eq:nabla:u}
\begin{split}
	\nabla u
	= & -\frac{u_r}{r^2}e_r\otimes e_r+(\frac{1}{r}\frac{\partial u_r}{\partial \theta}-\frac{u_{\theta}}{r})e_r\otimes e_{\theta}-\frac{u_{\phi}}{r}e_{r}\otimes e_{\phi}	- \frac{u_{\theta}}{r^2}e_{\theta}\otimes e_r \\
	& +(\frac{1}{r}\frac{\partial u_{\theta}}{\partial \theta}+\frac{u_{r}}{r})e_{\theta}\otimes e_{\theta} - \cot\theta\frac{u_{\phi}}{r}e_{\theta}\otimes e_{\phi}
	-\frac{u_{\phi}}{r^2}e_{\phi}\otimes e_r+\frac{1}{r}\frac{\partial u_{\phi}}{\partial \theta}e_{\phi}\otimes e_{\theta} \\
	& + (\cot\theta\frac{u_{\theta}}{r}+\frac{u_{r}}{r})e_{\phi}\otimes e_{\phi} \\
\end{split}
\end{equation}
Without loss of generality, consider $x\to S$ on $\mathbb{S}^2$ as $x'\to 0$. As explained in the proof of Theorem \ref{thm:expansion}, (\ref{eq:NS}) for $(-1)$-homogeneous axisymmetric solutions is reduced to (\ref{eq:NS:0}) for $y=\cos\theta$ and $U(y)=ur\sin\theta$, and the divergence free condition is equivalent to $ru_r=U'_{\theta}(y)$, where we use `` $'$ " to denote differentiation in $y$. Note Type 2 solutions satisfy $U_{\theta}(-1)=0$. By (\ref{eq:Uth:5}) and (\ref{eq:Uph:5}), $U_{\theta}=O(|x'|^2\ln |x'|)$ and $U_{\phi}=O(|x'|^2)$. Note $U'_{\theta}=ru_r=O(\ln |x'|)$ on $\bS^2$ by Theorem \ref{thm:expansion} (v). By the second equation in (\ref{eq:NS:0}), we have $U'_{\phi}=O(1)$. Note $r\sin\theta=|x'|$, we have 
\[
	u_{\theta}=\frac{U_{\theta}}{|x'|}=O(|x'| \ln |x'|), \quad u_{\phi}=\frac{U_{\phi}}{|x'|}=O(|x'|), \quad u_r=O(\ln |x'|).
\]
\[
	\frac{\partial u_{\theta}}{\partial\theta}=-\frac{1}{r^2}(U'_{\theta}+\frac{\cos\theta}{\sin^2\theta}U_{\theta})=O(\ln |x'|), \quad \frac{\partial u_{\phi}}{\partial \theta}=-\frac{1}{r^2}(U'_{\phi}+\frac{\cos\theta}{\sin^2\theta}U_{\phi})=O(1). 
\]
By the above and (\ref{eq:nabla:u}), we have 
\[
	\nabla u=\frac{1}{r}\frac{\partial u_r}{\partial \theta}e_r\otimes e_{\theta}+O(\ln |x'|)=-\frac{\sin\theta}{r^2}U''_{\theta}e_r\otimes e_{\theta}+O(\ln |x'|). 
\]
Differentiating the first equation in (\ref{eq:NS:0}), plugging the behavior of $U_{\theta}, U_{\theta}', U_{\phi}, U_{\phi}'$, it can be shown that $0<\limsup_{|x|=1, x'\to 0}|x'|^2|U''_{\theta}|<\infty$. Then the above implies $0<\limsup_{|x|=1,x'\to 0}|x'||\nabla u(x)|<\infty$, and $u$ is of Type 2'. 

Similarly, if a solution is of Type 3, consider $x\to S$ on $\mathbb{S}^2$. By Theorem \ref{thm:expansion}, we have $\tau=U_{\theta}(-1)=\lim_{|x|=1, x\to S}|x'|u_{\theta}\ne 0$. By Theorem 1.3 and 1.4 in \cite{LLY1}, we have the behavior of $U_{\theta}$ and $U_{\phi}$ corresponding to each $\tau$. By the second equation of (\ref{eq:NS:0}), one can obtain the behavior of $U_{\phi}'$. Taking derivative of the first equation in (\ref{eq:NS:0}), using the behavior of $U_{\theta}, U'_{\theta}, U_{\phi}$ and $U'_{\phi}$ for each $\tau$ respectively, the behavior of $U''_{\theta}$ can be obtained. Then with the estimation of $U_{\theta}, U_{\theta}', U_{\theta}'', U_{\phi}, U_{\phi}'$ and (\ref{eq:nabla:u}), we have that $u$ is of Type 3'. We omit the detail here. 

\subsection{Some examples of special solutions}

Below we discuss some special $(-1)$-homogeneous solutions of (\ref{eq:NS}) with isolated singularities on $\mathbb{S}^2$ and their asymptotic behavior. Due to the $(-1)$-homogeneity, we only consider the equations on $\mathbb{S}^2$ and all solution formulas in the following examples are given on $\mathbb{S}^2$. 

\subsubsection{Homogeneous axisymmetric solutions of (\ref{eq:NS})} 
 
In this section, we give some examples of $(-1)$-homogeneous axisymmetric solutions of (\ref{eq:NS}) in $C^2(\mathbb{S}^2\setminus\{S, N\})$ and discuss about their singularity behavior. The no-swirl solutions with one singularity at $S$ are classified in \cite{LLY1}, and the no-swirl solutions with two singularities at $S$ and $N$ are classified in \cite{LLY2}. The $(-1)$-homogeneous axisymmetric solutions with nonzero swirl nearby the no-swirl solutions surface in $C^2(\mathbb{S}^2\setminus\{S\})$ and $C^2(\mathbb{S}^2\setminus\{S, N\})$ were constructed respectively in \cite{LLY1} and \cite{LLY3}.
The asymptotic behavior of these solutions are described by Theorem \ref{thm:expansion}. 

\medskip

\noindent\textbf{(a) No-swirl solutions in $C^{\infty}(\mathbb{S}^2\setminus\{S\})$}

With one singularity on $\bS^2$, all $(-1)$-homogeneous axisymmetric no-swirl solutions in $C^2(\mathbb{S}^2\setminus\{S\})$ are classified in \cite{LLY1}. 

\begin{example}[\cite{LLY1}]\label{example:remove:S}
	Let $\mathcal{I} := \{(\tau,\sigma)\in\mathbb{R}^2\mid \tau\le 2, \sigma \le \frac{1}{4}(4-\tau)\}\cup \{(\tau,\sigma)\mid \tau\ge 2, \sigma =\frac{\tau}{4}\}$. Then for every $ (\tau, \sigma)\in \mathcal{I}$, there exists a unique $u_{\theta}\in C^{\infty}(\mathbb{S}^2\setminus\{S\})$ such that the corresponding $(u,p)$ is a solution to (\ref{eq:NS}) on $\mathbb{S}^2\setminus\{S\}$, satisfying $\lim_{\theta\to \pi^-} u_{\theta}\sin\theta=\tau$ and $\lim_{\theta\to 0^+}u_{\theta}/\sin\theta=\sigma$. Moreover, these are all the axisymmetric no-swirl solutions in $C^2(\mathbb{S}^2\setminus\{S\})$.	The explicit expressions of these solutions are as follows. 
	\begin{equation}
		u_\theta = 
		\left\{
		\begin{array}{ll}
			\frac{\displaystyle 1-\cos\theta}{\displaystyle \sin\theta}\left(1- b - \frac{\displaystyle 2 b(1-2\sigma-b)}{\displaystyle(1-2\sigma+b) (\frac{1+\cos\theta}{2})^{-b} + 2\sigma-1 + b} \right), & \tau<2, \\
			\frac{\displaystyle 1-\cos\theta}{\displaystyle \sin\theta} \left( 1+ \frac{\displaystyle 2 (1-2\sigma) }{ \displaystyle (1-2\sigma) \ln \frac{1+\cos\theta}{2} - 2 } \right), & \tau=2, \\
			\frac{\displaystyle (1+b)(1-\cos\theta)}{\displaystyle \sin\theta}, & \tau>2, 
		\end{array}
		\right. 
	\end{equation}
	where $(\tau, \sigma)\in \mathcal{I}$ and $b := |1-\frac{\tau}{2}|$, and $u_{r}$, $p$ can be determined by 
	\begin{equation}\label{eq:example:remove:S}
		u_r = - \frac{d u_{\theta}}{d \theta} - u_{\theta} \cot \theta, \quad p=u_r-\frac{1}{2}u_{\theta}^2 + \textrm{const}, 
		\quad \textrm{ on }\mathbb{S}^2. 
	\end{equation}
	Note the first equation of (\ref{eq:example:remove:S}) is equivalent to the divergence free condition $\dive u=0$ for $(-1)$-homogeneous axisymmetric solutions of (\ref{eq:NS}). 

	The solutions $u^{\tau, \sigma}$ are of Type 1 when $\tau=0$, and are of Type 3 when $\tau\ne 0$ regarding their behavior near the south pole. There is no Type 2 solution in $\{u^{\tau, \sigma}\mid (\tau, \sigma)\in \mathcal{I}\}$.
	
	To be precise, $\{(u_{\theta})_{\tau,\sigma} \mid \tau=0, \sigma\in (-\infty,0)\cup (0,1)\}$ are Landau solutions. They can also be rewritten as 
	\[
		u_{\theta}=\frac{2\sin\theta}{\frac{2-\sigma}{\sigma}+\cos\theta}. 
	\]
	Note when $\tau=2$, in view of (\ref{eq:example:remove:S}), we have 
	\[
		u_r = - \Bigg( 1+ \frac{ 2 ( 1 - 2 \sigma) }{ \displaystyle ( 1 - 2 \sigma) \ln \frac{ 1 + \cos \theta }{2} - 2 } \Bigg) - \frac{ 2 ( 1 - 2 \sigma )^2( 1 - \cos \theta ) }{ \Big( \displaystyle ( 1 - 2 \sigma ) \ln \frac{ 1 + \cos \theta }{2} - 2 \Big)^2 ( 1 + \cos \theta ) }. 
	\]
	So $\lim_{x\in \mathbb{S}^2, x\to S} |x'|^2 \big| \ln |x'| \big|^2 u_r = - 2$. By Theorem \ref{thm:expansion}, all $(-1)$-homogeneous axisymmetric solutions of (\ref{eq:NS}) satisfies $u=O(1/\big| |x'| \ln |x'| \big|^2)$. By the above example of solutions when $\tau=2$, this estimate is optimal.
\end{example}

\medskip

\noindent\textbf{(b) No-swirl solutions in $C^{\infty}(\mathbb{S}^2\setminus\{S, N\})$}

In \cite{LLY2}, all $(-1)$-homogeneous axisymmetric no-swirl solutions in $C^2(\mathbb{S}^2 \setminus \{ S, N\})$ were classified. For these solutions, $u_{\phi}=0$ and (\ref{eq:NS}) is reduced to (\ref{eq:NS:1}) for $y=\cos\theta$ and $U=u\sin\theta$ on $\bS^2$. The divergence free condition gives $u_{r} =dU_{\theta}/dy$, and $p$ can be determined by $p=u_r-u_{\theta}^2/2 + \text{const}$ on $\bS^2$. Denote $c := (c_1,c_2,c_3)$, and let
\[
	\bar{c}_3 := - \frac{1}{2} \left( \sqrt{1+c_1} + \sqrt{1+c_2} \right) \left( \sqrt{1+c_1} + \sqrt{1+c_2} + 2\right),
\] and 
\begin{equation}\label{eq:J}
	J:=\{c\in\mathbb{R}^3\mid c_1\ge -1, c_2\ge -1, c_3\ge \bar{c}_3\}.
\end{equation}
In \cite{LLY2}, it was proved that there exist $\gamma^-,\gamma^+\in C^0(J, \mathbb{R})$, satisfying $\gamma^-(c)<\gamma^+(c)$ if $c_3>\bar{c}_3(c_1,c_2)$, and $\gamma^-(c)=\gamma^+(c)$ if $c_3=\bar{c}_3(c_1,c_2)$, such that equation (\ref{eq:NS:1}) has a unique solution $U_{\theta}^{c,\gamma}$ in $C^{\infty}(-1,1)\cap C^0[-1,1]$ satisfying $U_{\theta}^{c,\gamma}(0)=\gamma$ for every $c \in J$ and $\gamma^-(c)\le \gamma \le \gamma^+(c)$. Define
\begin{equation*}
	I:= \{(c, \gamma)\in \mathbb{R}^4 \mid c_1\geq -1, c_2\geq -1, c_3\geq \bar{c}_3(c_1,c_2), \gamma^-(c)\leq \gamma \leq \gamma^+(c) \}, 
\end{equation*}
and
\begin{equation}
\begin{split}
	& u^{c,\gamma}:= u^{c,\gamma}_re_r+u^{c,\gamma}_{\theta}e_{\theta}=(\frac{d}{dy}U^{c,\gamma}_{\theta})e_r+\frac{U^{c,\gamma}_{\theta}}{\sin\theta}e_{\theta}, \\
	& p^{c,\gamma}:=u^{c,\gamma}_r-\frac{1}{2}(u^{c,\gamma}_{\theta})^2 + c_3 = \frac{d}{dy}U^{c,\gamma}_{\theta}-\frac{(U^{c,\gamma}_{\theta})^2}{2\sin^2\theta} + c_3.
\end{split}
\end{equation}
Then $\{(u^{c,\gamma}, p^{c,\gamma})\mid (c, \gamma)\in I\}$ are all the $(-1)$-homogeneous axisymmetric no-swirl solutions of (\ref{eq:NS}) in $C^{\infty}(\bS^2\setminus\{S, N\})$. 

In particular, it is obtained in \cite{LLY2} that $\tau:=U_{\theta}(-1)=2\pm 2\sqrt{1+c_1}$ and $\tilde{\tau}:=U_{\theta}(1)=-2\pm 2\sqrt{1+c_2}$, and the behavior of the solutions near $S$ and $N$ is described by Theorem \ref{thm:expansion} and its analogous result near $N$ for different values of $\tau$ and $\tilde{\tau}$. For $(c, \gamma)\in I$ with $\gamma^{-}(c)<\gamma<\gamma^{+}(c)$, the solutions $u^{c,\gamma}$ are of Type 1 if $c_1=c_2=c_3=0$, Type 2 if $c_1=c_2=0$, $c_3\ne 0$, and Type 3 if $c_1\ne 0$ or $c_2\ne 0$. If $\gamma=\gamma^{+}(c)$ or $\gamma^{-}(c)$, then $u^{c, \gamma}$ is of Type 3. 

\medskip

\noindent\textbf{(c) Homogeneous axisymmetric solutions with Type 2 singularities} 

Let us identify all Type 2 $(-1)$-homogeneous axisymmetric solutions in $C^2(\bS^2\setminus\{S, N\})$. We first consider no-swirl solutions. Let $\{(u^{c, \gamma}, p^{c, \gamma})\mid (c, \gamma)\in I\}$ be the $(-1)$-homogeneous axisymmetric no-swirl solutions of (\ref{eq:NS}) in $C^2(\bS^2\setminus\{S, N\})$ as described above. In \cite{LY} (Corollary 2.1), it is proved that if $c_1=c_2=0$ and $\lim_{x\in\mathbb{S}^2,x'\to 0} |x'|u_{\theta}^{c,\gamma} = 0$, then 
\begin{equation*}
	u^{c,\gamma}_{\theta}(x)=-\frac{c_3\mathrm{sgn}(x_3)|x'|}{|x|^2}\ln \frac{|x'|}{|x|}+\frac{O(1)(|c|+|\gamma|)|x'|}{|x|^2}, 
\end{equation*}
\begin{equation*}
	u^{c,\gamma}_{r}(x)=\frac{2c_3}{|x|}\ln\frac{|x'|}{|x|}+\frac{O(1)(|c|+|\gamma|)}{|x|}.
\end{equation*}
So $u^{c,\gamma}=u_re_r+u_{\theta}e_{\theta}$ satisfies 
\begin{equation}\label{eq:limit}
	\lim_{x\in \mathbb{S}^2, x\to S}\frac{|u^{c,\gamma}|}{\ln |x'|}= \lim_{x\in \mathbb{S}^2, x\to N}\frac{|u^{c,\gamma}|}{\ln |x'|}=-2|c_3|. 
\end{equation}
This in particular implies that for any $\alpha>0$, there exists a $(-1)$-homogeneous axisymmetric no-swirl solution $(u, p)\in C^{\infty}(\mathbb{S}^2\setminus\{S, N\})$ of (\ref{eq:NS}), such that $\lim_{|x'|\to 0}|u|/\ln |x'|=-\alpha$ on $\mathbb{S}^2$. As a consequence, the removable singularity result Theorem \ref{thm_main} is optimal. An explicit solution with $c_3=-4$ is given below in Example \ref{example:least:sing}. The asymptotic stability of such $(u^{c,\gamma}, p^{c, \gamma})$ was proved in \cite{LY}. 

Below we show that without the assumption of no-swirl, all Type 2 homogeneous axisymmetric solutions of (\ref{eq:NS}) on $\mathbb{S}^2\setminus\{S, N\}$ must be the $u^{c, \gamma}$ with $(c, \gamma)\in I$ and $c_1=c_2=0$. In particular, this implies that there are no Type 2 $(-1)$-homogeneous axisymmetric solutions of (\ref{eq:NS}) with a single singularity on $\mathbb{S}^2$. 

\begin{lem}\label{lem:3_1}
	Suppose $u\in C^2(\mathbb{S}^2\setminus\{S, N\})$ is a $(-1)$-homogeneous axisymmetric solution of (\ref{eq:NS}) satisfying $\limsup_{|x|=1,x'\to 0}|u(x)|/| \ln |x'||<\infty$. Then $u_{\phi}\equiv 0$, and $u$ must be the solutions $u^{c, \gamma}$ with $(c, \gamma)\in I$ satisfying $c_1=c_2=0$, and $u$ satisfies (\ref{eq:limit}). 
	In particular, let $P$ be $S$ or $N$, then there does not exist $(-1)$-homogeneous axisymmetric solution $u\in C^2(\mathbb{S}^2\setminus\{P\})$ of (\ref{eq:NS}) satisfying $0<\limsup_{|x|=1,x\to P}|u(x)|/| \ln |x'||<\infty$. 
\end{lem}
\begin{proof}
	Let $u\in C^2(\mathbb{S}^2\setminus\{S, N\})$ be a $(-1)$-homogeneous axisymmetric solution of (\ref{eq:NS}). As in the proof of Theorem \ref{thm:expansion}, (\ref{eq:NS}) is reduced to (\ref{eq:NS:0}) for $y=\cos\theta$ and $U=u\sin\theta$ on $\bS^2$. By the second equation in (\ref{eq:NS:0}), we have 
	\begin{equation*}
		\frac{d}{dy}U_{\phi}(y)=Ce^{-\int_0^y\frac{U_{\theta}(s)}{1-s^2}ds},
	\end{equation*}
	for some constant $C$. So $U_{\phi}$ is monotone in $y\in (-1, 1)$. By the assumption that $\limsup_{|x|=1,x'\to 0}|u(x)|/| \ln |x'||<\infty$, we have $U_{\phi}=u_{\phi}\sin\theta=O(\sin\theta\ln \sin\theta)$ on $\bS^2$. So $U_{\phi}(\pm 1)=0$. Then by the monotonicity of $U_{\phi}$, we must have $U_{\phi}\equiv 0$ and the system (\ref{eq:NS:0}) is further reduced to (\ref{eq:NS:1}). In view of the classification of all $(-1)$-homogeneous axisymmetric no-swirl solutions of (\ref{eq:NS}) on $\bS^2\setminus\{S, N\}$ in \cite{LLY2} (as described above), we have that $u=u^{c, \gamma}$ for some $(c, \gamma)\in I$ with $c_1=c_2=0$. In particular, $u$ satisfies (\ref{eq:limit}). 
 
	Next, for $P=S$ or $N$, we show there does not exist $(-1)$-homogeneous axisymmetric solution $u\in C^2(\mathbb{S}^2\setminus\{P\})$ satisfying $0<\limsup_{|x|=1,x\to P}|u(x)|/| \ln |x'||<\infty$. Without loss of generality, assume $P=S$. By the above argument, we have $u=u^{c, \gamma}$ with $c_1=c_2=0$ and satisfies (\ref{eq:NS:1}). Taking $d/dy$ of (\ref{eq:NS:1}) with $c_1=c_2=0$, we have 
	\[
		(1-y^2)\frac{d^2}{dy^2}U_{\theta}+2U_{\theta}+U_{\theta}\frac{d}{dy}U_{\theta}=-2c_3y. 
	\]
	Since $u$ is smooth at $N$, we have $U_{\theta}(1)=0$ and $d^2U_{\theta}/dy^2=O(1)$ as $y\to 1$. By this and the above equation, we must have $c_3=0$. Then solving (\ref{eq:NS:1}) on $(-1, 1]$ with $c_1=c_2=c_3=0$ and $U_{\theta}(1)=0$, we have $U_{\theta}=2(1-y^2)/(\lambda+y)$ for some $\lambda<-1$ or $\lambda\ge 1$. If $\lambda=1$, then $U_{\theta}=2(1-y)$ and $\limsup_{|x|=1,x\to S}|u(x)|/| \ln |x'||=\infty$. If $\lambda\ne 1$, then $u\in C^{\infty}(\mathbb{S}^2)$, which is a Landau solution. Thus there is no solution $u\in C^2(\mathbb{S}^2\setminus\{S\})$ satisfying $0<\limsup_{|x|=1,x\to S}|u(x)|/| \ln |x'||<\infty$. The proof is finished. 
\end{proof}
Below we give some examples of Type 2 solution with explicit formulas, which exist near $S$ or $N$ on $\mathbb{S}^2$. 

\begin{example}\label{example:least:sing}
	In equation (\ref{eq:NS:1}), set $c_1=c_2=0$. Let $U_\theta (y)= 2(1-y^2)\frac{\chi'(y)}{\chi(y)}$ for some function $\chi(y)$, then (\ref{eq:NS:1}) is converted to 
	\[
		2(1-y^2) \chi''(y) - c_3 \chi(y) = 0. 
	\]
	Let $z := \cos^2\frac{\theta}{2}$, then $y = 2z-1$. Denote $\tilde{\chi}(z) := \chi(y(z)) $. The above equation then becomes
	\[
		2 z (1-z) \frac{d^2 \tilde{\chi}(z)}{dz^2} - c_3 \tilde{\chi}(z) = 0. 
	\]
	This is a hypergeometric differential equation, whose solution $\tilde{\chi}$ can be expressed with the help of hypergeometric functions and Meijer G-functions \cite{OLBC,EMOT}. In general, these functions can not be expressed as elementary functions. According to the above argument about $u^{c, \gamma}$ and Theorem \ref{thm:expansion}, if $c_3>\bar{c}_3(0,0) = -4$, then the solution lies in $C^2(\mathbb{S}^2\setminus\{S, N\})$, otherwise there exist local solutions in $C^2(\mathbb{S}^2\cap B_\delta(S)\setminus\{S\})$ satisfying (\ref{eq:limit}) for some $\delta > 0$. 

	In particular, when $c_3=-4$, there are special solutions explicitly given by 
	\[
		\chi(y) = (1-y^2) \ln ( \frac{1+y}{1-y} ) + 2y + 2\alpha (1-y^2)
	\]
	for any $\alpha\in \R$. Correspondingly, 
	\begin{equation}\label{eq:example:least:sing}
	\left\{
	\begin{aligned}
		& u_\theta = 4\sin\theta \cdot \frac{ 1 - \cos\theta ( \ln( \cot\frac{ \theta}{2} ) + \alpha ) }{ \cos\theta + \sin^2\theta ( \ln( \cot\frac{ \theta}{2} ) + \alpha ) }, \\
		& u_r = -4 - \frac{ 8 \big( 1 - \cos\theta (\ln( \cot\frac{\theta}{2} ) + \alpha ) \big) }{ \big( \cos\theta + \sin^2\theta (\ln( \cot\frac{ \theta}{2} ) + \alpha ) \big)^2 }, \\
		& u_\phi = 0, \\
		& p = \frac{ 8 \big( - 2 + \cos\theta (\ln\cot\frac{\theta}{2} + \alpha ) - \sin^2\theta ( \ln \cot\frac{\theta}{2} + \alpha )^2 \big) }{ \big( \cos\theta + \sin^2\theta ( \ln \cot\frac{ \theta}{2} + \alpha ) \big)^2 }
	\end{aligned}
	\right.
	\end{equation}
	is a special solution of (\ref{eq:NS}) satisfying (\ref{eq:limit}) for $\theta \in (\theta_0,\pi)$, where $\theta_0$ is the unique root of $\cos\theta + \sin^2\theta ( \ln \cot\frac{ \theta}{2} + \alpha )=0$. Thus (\ref{eq:example:least:sing}) is a solution of (\ref{eq:NS}) in $C^2(\mathbb{S}^2 \cap B_\delta(S)\setminus\{S\})$ for some $\delta>0$ depending on $\alpha$. Note that (\ref{eq:example:least:sing}) is also a solution for $\theta \in (0, \theta_0)$, thus is a solution in $C^2(\mathbb{S}^2 \cap B_{\delta'}(N)\setminus\{N\})$ for some $\delta'>0$. This special solution also implies that our removable singularity result Theorem \ref{thm_main} is optimal. Note this solution does not exist on the whole $\mathbb{S}^2 \setminus\{S,N\}$. Indeed, in this example, $c_1=c_2=0$ and $c_3=-4=\bar{c}_3(0,0)$. By Theorem 1.1 in \cite{LLY2}, in this case there exists only one solution of (\ref{eq:NS}) in $C^2(\mathbb{S}^2 \setminus\{S,N\})$, which is given by $u_{\theta}=-4\cot\theta$, $u_{\phi}=0$, $u_r=-4$ and $p=-8\csc^2\theta$. 

	When $c_1=c_2=0$, $c_3=1/2>-4$, there are special solutions in $C^2(\mathbb{S}^2\setminus\{S, N\})$ with explicit but not elementary expressions given by
	\[
	\left\{
	\begin{aligned}
		& u_\theta = \frac{ \sin\theta \big( K( \cos^2\frac{\theta}{2} ) - \alpha K( \sin^2\frac{\theta}{2} ) \big) }{2 \big( E( \cos^2\frac{\theta}{2} ) - \sin^2\frac{\theta}{2} K( \cos^2\frac{\theta}{2} ) + \alpha E( \sin^2\frac{\theta}{2} ) - \alpha \cos^2\frac{\theta}{2} K( \sin^2\frac{\theta}{2} ) \big) }, \\
		& u_r = \frac{-1}{\sin\theta}\frac{d}{d\theta}(\sin\theta u_\theta), \\
		& u_\phi = 0, \\
		& p = \frac{-1}{2} \Big( \frac{d^2 u_r}{d\theta^2} + (\cot\theta-u_\theta) \frac{d u_r}{d\theta} + u_r^2 +u_\theta^2 \Big), 
	\end{aligned}
	\right.
	\]
	for any $0<\alpha <+\infty$, where $K(x)$ and $E(x)$ are respectively the complete elliptic integrals of the first and second kind
	\[
		K(x) = \int_0^{\pi/2} \frac{1}{\sqrt{1-x \sin^2\theta} } d \theta, \qquad 
		E(x) = \int_0^{\pi/2} \sqrt{1-x \sin^2\theta d\theta}, \qquad 0<x<1. 
	\]
	These special solutions lie in $C^2(\mathbb{S}^2\setminus\{S, N\})$ and satisfy (\ref{eq:limit}), which also implies that the removable singularity result Theorem \ref{thm_main} is optimal. For $\alpha=0$ or $+\infty$, the above solution is a solution of (\ref{eq:NS}) in $C^2(\mathbb{S}^2\setminus\{S, N\})$ of Type 3. If $\alpha<0$, then the above solution is a local solution near $N$ or $S$ on $\mathbb{S}^2$. 
\end{example}

\subsubsection{Serrin's solutions}

In a pioneering work \cite{Serrin} concerning a representative model for tornadoes, Serrin studied $(-1)$-homogeneous axisymmetric solutions of (\ref{eq:NS}) in the upper half space $\mathbb{R}^3_+:=\mathbb{R}^3\cap\{x_3>0\}$ with a singular ray along the positive $x_3$-axis and some boundary conditions on $\partial \mathbb{R}^3_+$. Notably, the solutions formulated in the paper exhibit properties different from Landau solutions, where $u_{\theta}=O(|x'| \ln |x'|)$, $u_r=O(\ln |x'|)$ near the north pole $N$ on $\mathbb{S}^2$, and $0<\lim_{x\in \mathbb{S}^2, x\to N}|x'| |u_{\phi}|<\infty$. In particular, Serrin's solutions are of Type 3 behavior mentioned above. 

\subsubsection{Solutions given by Liouville formulas}

In \cite{Sverak}, \v{S}ver\'{a}k proved that all $(-1)$-homogeneous nonzero solutions of (\ref{eq:NS}) in $C^2(\mathbb{R}^3\setminus\{0\})$ are Landau solutions (see Theorem \ref{thm:Sverak}). In his proof, (\ref{eq:NS}) is reduced to 
\begin{equation}\label{eq:liouville:S2}
	-\triangle_{\mathbb{S}^2} \varphi + 2 = 2 e^{\varphi}, \quad \mbox{on } \mathbb{S}^2, 
\end{equation}
where $u=\nabla_{\mathbb{S}^2} \varphi-\Delta_{\mathbb{S}^2} \varphi e_r$ on $\mathbb{S}^2$. It is clear from \cite{Sverak} that if $\varphi$ is a solution to (\ref{eq:liouville:S2}), then $u=\nabla_{\mathbb{S}^2} \varphi-\Delta_{\mathbb{S}^2} \varphi e_r$ is a solution to (\ref{eq:NS}) after being extended to a $(-1)$-homogeneous vector field in $\mathbb{R}^3$. 

The classification of solutions to (\ref{eq:liouville:S2}) is classical. Let $F^{-1}: \mathbb{S}^2\to \R^2$ be the stereographic projection, with $z=(z^1, z^2)=F^{-1}(x)$ given by $z^i = x_i/(1-x_3)$, $i=1, 2$. It is easy to check that for any bounded open set $\mathcal{O}\subset \mathbb{R}^2$, $\xi\in C^2(\mathcal{O})$ is a solution to 
\begin{equation}\label{eq:liouville:R2}
	- \triangle \xi = e^\xi 
\end{equation}
in $\mathcal{O}$, if and only if 
\begin{equation}\label{eq:liouville:phi} 
	\varphi(x):= \xi \circ F^{-1}(x) - 3\ln 2 + 2 \ln (1+| F^{-1}(x) |^2) 
\end{equation}
is a solution to (\ref{eq:liouville:S2}) in $F(\mathcal{O})\subset \mathbb{S}^2$. 

For a simply connected open set $\mathcal{O}\subset \mathbb{R}^2$, it is known (see \cite{Liouville} and \cite{CW}) that all real solutions $\xi\in C^2(\mathcal{O})$ of (\ref{eq:liouville:R2}) are of the form 
\begin{equation}\label{eq:liouville:xi}
	\xi = \ln \frac{8 |f'(z)|^2}{(1+|f(z)|^2)^2}, 
\end{equation}
with $f$ being a locally univalent meromorphic function. Here we have abused notations slightly by identifying $z = z^1 + i z^2$ and using $\mathcal{O}$ also to denote the subset $\{z^1 + i z^2 \mid (z^1, z^2) \in \mathcal{O}\} \subset \mathbb{C}$. In particular, if $\xi$ is singular at some $z$, then $\varphi$ is singular at $F(z)$ on $\mathbb{S}^2$, and the corresponding $u=\nabla_{\bS^2} \varphi-\Delta_{\bS^2} \varphi e_r$ is a solution to (\ref{eq:NS}) with a singularity at $F(z)$ on $\mathbb{S}^2$. In view of this fact, we may construct some special $(-1)$-homogeneous solutions of (\ref{eq:NS}) with arbitrary finite singularities on $\mathbb{S}^2$. 

First, for axisymmetric solutions in $C^2(\mathbb{S}^2\setminus\{N,S\})$, we have

\begin{lem}
	Let $u$ be a $(-1)$-homogeneous solution of (\ref{eq:NS}) given by $u=\nabla_{\bS^2} \varphi-\Delta_{\bS^2} \varphi e_r$ on $\mathbb{S}^2\setminus\{S, N\}$, with $\varphi$ given by (\ref{eq:liouville:phi}) and (\ref{eq:liouville:xi}) for some multi-valued locally univalent meromorphic function $f$ on $\mathbb{C}\setminus\{0\}$. If $u$ is axisymmetric, then $f=a z^{\alpha}$ for some $a\in \mathbb{C}\setminus\{0\}$ and $\alpha\in\mathbb{R}\setminus\{0\}$. 
\end{lem}
\begin{proof}
	Since $u=\nabla_{\bS^2} \varphi-\Delta_{\bS^2} \varphi e_r$ is axisymmetric, $\varphi$ is also axisymmetric. So $\xi$ defined by (\ref{eq:liouville:xi}) is radially symmetric and satisfies (\ref{eq:liouville:R2}) in $\mathbb{R}^2\setminus\{0\}$. Denote $r=|z|$, we have $\xi=\xi(r)$. Let $t=\ln |z|$ and $\eta(t)=\xi(e^t)+2t$, then (\ref{eq:liouville:R2}) is reduced to 
	\[
		- \eta_{tt}=e^{\eta}. 
	\]
	Multiplying both sides of the above by $\eta_t$ and taking the integral, we have $\eta_t^2+2e^{\eta}=const$. Solving for $\eta$, we obtain 
	\[
		\eta=\ln \frac{Ae^{ct}}{(1+Ae^{ct})^2}+C
	\]
	for some constants $A>0, c>0$ and $C$. Then 
	\[
		\xi=\ln\frac{Ar^{c-2}}{(1+Ar^c)^2}+C, \quad \forall r>0.
	\]
	Note this $\xi$ is given by (\ref{eq:liouville:xi}) with $f(z)=\sqrt{A}z^{c/2}$ or $f(z)=z^{-c/2}/\sqrt{A}$. The lemma is proved.
\end{proof}

Now we display some special solutions of (\ref{eq:NS}) given by $u=\nabla_{\bS^2} \varphi-\Delta_{\bS^2} \varphi e_r$ , where $\varphi$ is given by (\ref{eq:liouville:phi}) and (\ref{eq:liouville:xi}) with some locally univalent meromorphic function $f$. 

\begin{example}
	In this example, we construct some special $(-1)$-homogeneous solutions of (\ref{eq:NS}) with singularities at $N$ or $S$ on $\bS^2$ using Liouville formulas, including both axisymmetric and non-axisymmetric solutions. For each $f(z)$, we define $\varphi$ by (\ref{eq:liouville:phi}) and (\ref{eq:liouville:xi}), and construct a corresponding solution of (\ref{eq:NS}) given by $u=\nabla_{\bS^2} \varphi-\Delta_{\bS^2} \varphi e_r$. We display the expression of $u=u_re_r+u_{\theta}e_{\theta}+u_{\phi}e_{\phi}$ corresponding to each $f(z)$ below. 

	\noindent\textbf{\em (a)}
	Take $f(z)=az^{\alpha}$ for some $a\in\mathbb{C}$, $0<|a|<+\infty$ and $\alpha\in\mathbb{R}$, $\alpha\ne 0$. The corresponding solution $(u, p)$ of (\ref{eq:NS}) is 
	\begin{equation}\label{eq:liouville:u}
	\left\{
	\begin{aligned}
		& u_{\theta}=u_{\theta}^{\alpha,|a|} := \frac{2}{\sin\theta} \Big( - \cos\theta + \alpha \frac{ |a|^2\cot^{2\alpha}\frac{\theta}{2} -1 }{ |a|^2\cot^{2\alpha}\frac{\theta}{2}+1} \Big), \\
		& u_{\phi}=0,\\
		& u_r=u_r^{\alpha,|a|} := - 2 + \frac{8\alpha^2}{\sin^2\theta} \cdot \frac{|a|^2 \cot^{2\alpha}\frac{\theta}{2}}{ (1 + |a|^2\cot^{2\alpha}\frac{\theta}{2})^2},\\
	\end{aligned}
	\right.
	\end{equation}
	\begin{equation}
		p=p^{\alpha,|a|}:= u_r - \frac{1}{2} (u_{\theta})^2.
	\end{equation}
	It is obvious that $u^{\alpha,|a|} = u^{-\alpha,1/|a|}$. Landau solutions correspond to the case when $\alpha=\pm 1$. When $\alpha\ne \pm 1$, the above solutions are of Type 3. 
	
	\noindent\textbf{\em (b)} Take $f = a e^{bz}$ for some $a,b\in\mathbb{C}$ satisfying $|a|,|b|\in (0,+\infty)$. The corresponding solution $(u, p)$ of (\ref{eq:NS}) is 
	\[
	\left\{
	\begin{aligned}
		& u_{\theta} = \frac{-2(\cos\theta+1)}{\sin\theta} + \frac{ (b_1 \cos\phi - b_2 \sin\phi) }{ \sin^2\frac{\theta}{2} } \cdot \tanh \Big( \cot(\frac{\theta}{2}) (b_1\cos\phi - b_2\sin\phi) + \ln |a| \Big), \\
		& u_{\phi} = 2 (b_1\sin\phi+b_2\cos\phi)\frac{\cot\frac{\theta}{2}}{\sin\theta} \cdot \tanh \Big( \cot(\frac{\theta}{2}) (b_1\cos\phi - b_2\sin\phi) + \ln |a| \Big), \\
		& u_r = -2 + \frac{|b|^2 }{ 2 \sin^4\frac{\theta}{2} } \cdot \text{sech}^2 \Big( \cot\frac{\theta}{2}(b_1\cos\phi - b_2 \sin\phi) + \ln |a| \Big), 
	\end{aligned}
	\right.
	\]
	and 
	\begin{equation}\label{eq:liouville:p}
		p = -\frac{1}{2} \Big( \frac{d^2 u_r}{d\theta^2} + \cot\theta \frac{d u_r}{d\theta} + \frac{1}{\sin^2\theta} \frac{d^2 u_r}{d\phi^2} - u_\theta \frac{d u_r}{d\theta} - \frac{u_\phi}{\sin\theta} \frac{d u_r}{d\phi} + u_r^2 + u_\theta^2 + u_\phi^2 \Big). 
	\end{equation}
	This solution $u$ is not axisymmetric. 	

	\noindent\textbf{\em (c)} Take $f = e^{z^k}$. Then the corresponding solution of (\ref{eq:NS}) is given by
	\begin{equation}
	\left\{
	\begin{aligned}
		& u_\theta = \frac{-2(\cos\theta+k)}{\sin\theta} + 2k \cos(k\phi) \frac{\cot^k\frac{\theta}{2}}{\sin\theta} \tanh \big( \cot^k(\frac{\theta}{2}) \cos(k\phi) \big), \\
		& u_\phi = 2k \sin(k\phi) \frac{\cot^k(\frac{\theta}{2})}{\sin\theta} \tanh \big( \cot^k(\frac{\theta}{2}) \cos(k\phi) \big), \\
		& u_r = -2 + 2k^2 \frac{\cot^{2k}(\frac{\theta}{2})}{\sin^2\theta} \emph{sech}^2 ( \cot^k(\frac{\theta}{2}) \cos (k\phi) ). 
	\end{aligned}
	\right.
	\end{equation}
	The pressure $p$ can be derived from (\ref{eq:liouville:p}), and the corresponding $u$ is not axisymmetric. 	
\end{example}

\vspace{0.5cm}

It should be noted that if we send $|a|\to 0$ in (\ref{eq:liouville:u}), then 
\begin{equation}\label{eq:liouville:limit}
	\lim_{|a|\to 0}u_{\theta}^{\alpha,|a|}= \frac{-2(\alpha + \cos\theta)}{\sin\theta}, 
\end{equation}
For $\alpha=\pm 1$ and $0<|a|<+\infty$, $u^{\alpha, |a|}$ is a Landau solution. When $\alpha=1$, the limit in (\ref{eq:liouville:limit}) gives a solution $u$ of (\ref{eq:NS}) with 
\[
	u_{\theta}=-\frac{ 2(1+\cos\theta)}{\sin\theta},\quad u_r=-2, \quad u_{\phi}=0, \quad p=-\frac{4(1+\cos\theta)}{\sin^2\theta}, 
\]
which is smooth at south pole and singular at north pole. This solution is given by $u=\nabla_{\bS^2} \varphi-\Delta_{\bS^2} \varphi e_r$ with $\varphi=-2\ln (1-\cos \theta)$, which satisfies $\Delta_{\bS^2}\varphi+2=0$. This $\varphi$ is not given by (\ref{eq:liouville:phi}) and (\ref{eq:liouville:xi}) with any locally univalent meromorphic function $f$. Similar situation holds for $\alpha=-1$, where the corresponding solution is
\[
	u_{\theta}=\frac{2(1-\cos\theta)}{\sin\theta}, \quad u_r=-2, \quad u_{\phi}=0, \quad p=-\frac{4(1-\cos\theta)}{\sin^2\theta},
\]
which is given by $u=\nabla_{\bS^2} \varphi-\Delta_{\bS^2} \varphi e_r$ with $\varphi=-2\ln (1+\cos\theta)$. This solution is smooth at north pole and singular at south pole. These solutions are of Type 3. In particular, these solutions also satisfy Euler's equations. 

We may also construct $(-1)$-homogeneous solutions of (\ref{eq:NS}) with singularities on $\bS^2$ that are not $N$ or $S$.

\begin{thm}\label{thm:3_1}
	Let $m\ge 2$ be an integer, $P_1, ..., P_{m}\in \mathbb{S}^2$ be distinct points, and $\{l_1, ..., l_{m}\}\in\mathbb{Z}\setminus\{0, 1, -1\}$ satisfy $\sum_{j=1}^{m}l_j=m-2$. Then there exists a $(-1)$-homogeneous solution $u\in C^{\infty}(\bS^2\setminus\{P_1, ..., P_{m}\})$ of (\ref{eq:NS}), satisfying 
	\begin{equation}\label{eq:thm:3_1} 
		u=2(|l_j|-1)\nabla_{\bS^2}\ln |x-P_j|+O(1), \quad \textrm{as }x\to P_j \textrm{ on }\bS^2, 
		\quad \forall 1\le j\le m.
	\end{equation}
\end{thm}
\begin{proof}
	By rotation of the coordinates, let $P_{m}$ be the north pole $N$. Let $F^{-1}: \mathbb{S}^2\to \R^2$ be the stereographic projection and $z_j:=F^{-1}(P_j)$, $1\le j\le m-1$. Fix any $a\in \mathbb{C}\setminus \{z_1, ..., z_{m-1}\}$ and define 
	\begin{equation}\label{eq:thm:3_1:f}
		f(z):=\int_a^z (t-z_1)^{l_1-1}\cdots (t-z_{m-1})^{l_{m-1}-1} dt, \quad \forall z\in \mathbb{C}\setminus\{z_1, ..., z_{m-1}\},
	\end{equation}
	where the integral path from $a$ to $z$ does not intersect with $\{z_1, ..., z_{m-1}\}$. Since $l_j\ne 0, 1$, $f$ is independent of the path and well-defined in $\mathbb{C}\setminus\{z_1, ..., z_{m-1}\}$. So $f$ is a locally univalent meromorphic function near each $z_j$, $1\le j\le m-1$. 

	Let $\xi(z)$ be defined by (\ref{eq:liouville:xi}) with this $f$, and $\varphi(x)$ be defined by (\ref{eq:liouville:phi}). Then as mentioned earlier, $\xi\in C^{\infty}(\mathbb{C}\setminus\{z_1, ..., z_{m-1}\})$ satisfies (\ref{eq:liouville:R2}) in $\mathbb{C}\setminus\{z_1, ..., z_{m-1}\}$, and $\varphi\in C^{\infty}(\bS^2\setminus\{P_1, ..., P_{m}\})$ satisfies (\ref{eq:liouville:S2}) on $\mathbb{S}^2 \setminus \{P_1, ..., P_{m}\}$. Let $u:=\nabla_{\bS^2} \varphi-\Delta_{\bS^2} \varphi e_r$ on $\mathbb{S}^2\setminus\{P_1, ..., P_{m}\}$ and be extended as a $(-1)$-homogeneous vector field in $\R^3$. Then $u\in C^{\infty}(\bS^2\setminus\{P_1, ..., P_{m}\})$ is a $(-1)$-homogeneous solution of (\ref{eq:NS}) on $\bS^2\setminus\{P_1, ..., P_{m}\}$.

	Now we prove that $u$ satisfies (\ref{eq:thm:3_1}) for all $1\le j\le m$. Let $(r, \theta, \phi)$ be spherical coordinates as usual. Write $x=(\sin\theta\cos\phi, \sin\theta\sin\phi, \cos\theta)^T$ for $x\in\bS^2$ and $z=F^{-1}(x)=\frac{\sin\theta}{1-\cos\theta}(\cos\phi, \sin\phi)^T$. By computation, 
	\begin{equation}\label{eq:thm:3_1:d}
	\begin{split}
		& \partial_{\theta}x=e_{\theta}, \quad \partial_{\phi}x=\sin\theta e_{\phi}, \\
		& \partial_{\theta}z=-\frac{1}{1-\cos\theta}(\cos\phi, \sin\phi)^T, \quad \partial_{\phi}z=\frac{\sin\theta}{1-\cos\theta}(-\sin\phi, \cos\phi)^T. 
	\end{split}
	\end{equation}
	For convenience, denote $A:=(\cos\theta-1)[\partial_\theta z, \partial_{\phi} z/\sin\theta]=
	\begin{bmatrix}
	\cos\phi & \sin\phi \\
	\sin\phi & -\cos\phi
	\end{bmatrix}$. 

	\noindent\textbf{Case 1.} $j\ne m$. 

	We only need to prove it for $j=1$. For $f$ given by (\ref{eq:thm:3_1:f}), we have 
	\begin{equation}\label{eq:thm:3_1:f:prime}
		f'(z)=\prod_{j=1}^m(z-z_j)^{l_j-1}= C l_1(z-z_1)^{l_1-1}+O(|z-z_1|^{l_1}) 
	\end{equation}
	and 
	\begin{equation}\label{eq:thm:3_1:f:2}
		f(z)=C(z-z_1)^{l_1}+O(|z-z_1|^{l_1+1}) + O(1)
	\end{equation}
	as $z\to z_1$, where $C = \frac{1}{l_1}\prod_{j=2}^m(z_1-z_{j})^{l_{j}-1}$. Denote
	\begin{equation}\label{eq:thm:3_1:xi:2}
		\hat{\xi}(z):=\xi(z)-3\ln 2+2\ln (1+|z|^2)=\ln \frac{|f'(z)|^2(1+|z|^2)^2}{(1+|f(z)|^2)^2}.
	\end{equation}
	By the definition of $\varphi$, we have $\varphi(x)=\hat{\xi}\circ F^{-1}(x)$, $x\in\mathbb{S}^2$. We first claim that 
	\begin{equation}\label{eq:thm:3_1:xi:3}
		\hat{\xi}=2(|l_1|-1)\ln |z-z_1|+O(1), \quad \textrm{as }z\to z_1.
	\end{equation}
	Indeed, by (\ref{eq:thm:3_1:f:prime}) and (\ref{eq:thm:3_1:xi:2}), we have 
	\begin{equation}\label{eq:thm:3_1:xi:4}
		\hat{\xi}= \ln \frac{|z-z_1|^{2(l_1-1)}}{(1+|f(z)|^2)^2}+O(1), 
		\quad \textrm{ as }z\to z_1.
	\end{equation}
	Note $l_1\ne 0, 1, -1$. If $l_1\ge 2$, (\ref{eq:thm:3_1:xi:3}) directly follows from (\ref{eq:thm:3_1:f:2}) and (\ref{eq:thm:3_1:xi:4}). If $l_1\le -2$, then by (\ref{eq:thm:3_1:f:2}) and (\ref{eq:thm:3_1:xi:4}), we have
	\[
		\hat{\xi}(z)=\ln \frac{|z-z_1|^{2(l_1-1)}|z-z_1|^{-4l_1}}{(1+|f(z)|^2)^2|z-z_1|^{-4l_1}}+O(1)
		= - 2 (l_1+1)\ln |z-z_1|+O(1).
	\]
	So (\ref{eq:thm:3_1:xi:3}) holds. 

	Next, let $\zeta(z):=\hat{\xi}(z)-2(|l_1|-1)\ln |z-z_1|$ and $\delta>0$ be small enough such that $z_j\notin B_{\delta}(z_1)$ for $j\ne 1$. Then by (\ref{eq:liouville:R2}), (\ref{eq:thm:3_1:xi:2}) and (\ref{eq:thm:3_1:xi:3}), we have 
	\[
		-\Delta\zeta=-\Delta\hat{\xi}= \frac{8(e^{\hat{\xi}}-1)}{(1+|z|^2)^2}=O(1), \quad \textrm{ in }B_\delta(z_1).
	\]
	Note $\zeta=O(1)$ in $B_{\delta}(z_1)$. By elliptic estimate, we have $\nabla \zeta=O(1)$ in $B_{\delta}(z_1)$. Thus 
	\begin{equation}\label{eq:thm:3_1:nabla:xi}
		\nabla \hat{\xi}=2(|l_1|-1)\nabla \ln |z-z_1|+\nabla \zeta=2(|l_1|-1)\frac{z-z_1}{|z-z_1|^2}+O(1), \quad \textrm{ as }z\to z_1. 
	\end{equation}
	Recall $u=\nabla_{\bS^2} \varphi-\Delta_{\bS^2} \varphi e_r$. To prove (\ref{eq:thm:3_1}), we will show $\Delta_{\bS^2} \varphi =O(1)$ and $\nabla_{\bS^2} \varphi=2(|l_1|-1)\nabla_{\bS^2}\ln |x-P_1|+O(1)$ as $x\to P_1$ on $\bS^2$. 

	Note $\varphi(x)=\hat{\xi}(z)$ and $x=F(z)\to P_1$ ans $z\to z_1$. By (\ref{eq:liouville:S2}), (\ref{eq:thm:3_1:xi:3}) and the fact $l_1\ne 0$, we have $-\Delta_{\bS^2}\varphi=2(e^{\varphi}-1)=O(1)$ as $x\to P_1$. 
	
	Next, we estimate $\nabla_{\bS^2}\varphi$ as $x\to P_1$ from all directions on $\bS^2$. For each fixed $z_0\ne z_1$, let $z=\tilde{\gamma}(t):=tz_0+(1-t)z_1$, $0\le t\le 1$, and $\gamma(t)=F(\tilde{\gamma}(t))$. Using (\ref{eq:thm:3_1:d}), it can be verified
	\begin{equation}\label{eq:thm:3_1:nabla:phi}
		\nabla_{\bS^2}\varphi=\partial_{\theta}\varphi e_{\theta}+\frac{1}{\sin\theta}\partial_{\phi}\varphi e_{\phi}
		=\nabla\hat{\xi}\cdot \partial_{\theta}z e_{\theta}+\frac{1}{\sin\theta}\nabla\hat{\xi}\cdot\partial_{\phi}z e_{\phi}
		=-\frac{1}{1-\cos\theta}[e_{\theta}, e_{\phi}] A^T \nabla\hat{\xi}. 
	\end{equation}
	On the other hand, write $x-P_1=\gamma(t)-\gamma(0)=\gamma'(t)t+O(t^2)$. By (\ref{eq:thm:3_1:d}), we have $ \gamma'(t) =\partial_{\theta}x\theta'(t)+\partial_{\phi}x\phi'(t)=[e_\theta, e_{\phi}]
	\begin{bmatrix}
	\theta'(t) \\
	\sin\theta\phi'(t)
	\end{bmatrix}$, and 
	\[
	\begin{split}
		z-z_1 & = (z_0-z_1)t=\tilde{\gamma}'(t)t=(\theta'(t)\partial_{\theta}z+\phi'(t)\partial_{\phi}z)t 
		=-\frac{t}{1-\cos\theta}A
		\begin{bmatrix}
			\theta'(t)\\
			\sin\theta\phi'(t)
		\end{bmatrix}\\
		& =-\frac{1}{1-\cos\theta}A
		\begin{bmatrix}
			e_{\theta}^T\\
			e_{\phi}^T
		\end{bmatrix}\gamma'(t)t=-\frac{1}{1-\cos\theta}A
		\begin{bmatrix}
			e_{\theta}^T\\
			e_{\phi}^T
		\end{bmatrix}(x-P_1)+O(t^2).
		\end{split}
	\]
	Using (\ref{eq:thm:3_1:nabla:phi}), (\ref{eq:thm:3_1:nabla:xi}), the above, the fact that $A^TA=I$, 
	we have 
	\[
	\begin{aligned}
		\nabla_{\bS^2}\varphi & =-\frac{2(|l_1|-1)}{1-\cos\theta}[e_{\theta}, e_{\phi}]A^T\frac{z-z_1}{|z-z_1|^2}+O(1) \\
		& =2(|l_1|-1)[e_{\theta}, e_{\phi}] 
		\begin{bmatrix}
		e_{\theta}^T \\
		e_{\phi}^T
		\end{bmatrix}\frac{x-P_1}{|x-P_1|^2}+O(1).
	\end{aligned}
	\]
	By (\ref{eq:thm:3_1:d}), 
	\begin{equation*}
	\begin{split}
		\nabla_{\bS^2}\ln |x-P_j| & =\partial_{\theta}\ln |x-P_j| e_{\theta}+\frac{1}{\sin\theta}\partial_{\phi}\ln |x-P_j| e_{\phi}\\
		& =\frac{(x-P_j)\cdot \partial_{\theta}x}{|x-P_j|^2}e_{\theta}+\frac{(x-P_j)\cdot \partial_{\phi}x}{\sin\theta|x-P_j|^2}e_{\phi}=[e_{\theta}, e_{\phi}] \begin{bmatrix}
		e_{\theta}^T\\
		e_{\phi}^T
		\end{bmatrix}\frac{x-P_j}{|x-P_j|^2}. 
	\end{split}
	\end{equation*}
	So we have proved (\ref{eq:thm:3_1}) in Case 1. 

	\noindent\textbf{Case 2.} $j=m$.
	
	Similar to Case 1, we only need to estimate $\Delta_{\bS^2}\varphi$ and $\nabla_{\bS^2}\varphi$ as $x\to P_{m}$. In this case, $P_{m}=N$. As $x\to N$, $|z|\to \infty$. Let $l=\sum_{j=1}^{m-1}l_j-m+2$. By the definition of $f$, we have
	\begin{equation*}
		f'(z)=z^{l-1}+O(|z|^{l-2})\quad \textrm{and }\quad f(z)= \frac{1}{l}z^{l}+O(|z|^{l-1})+O(1) 
	\end{equation*}
	as $|z|\to \infty$. Plug this into (\ref{eq:thm:3_1:xi:2}), we have 
	\begin{equation}\label{eq:thm:3_1:xi:5}
		\hat{\xi}=(2-2|l|)\ln |z|+O(|z|^{-1}). 
	\end{equation}
	Since $|l|\ge 2$, we have $e^{\hat{\xi}}=|z|^{2-2|l|}e^{O(1/|z|)}=O(1)$. So $\Delta_{\bS^2}\varphi=2(e^{\varphi}-1)=O(1)$.

	Now we estimate $\nabla_{\bS^2}\varphi$. Fix $z_0$ such that $\rho:=|z_0|/3>>1$ and $z_j\notin B_1(z_0)$ for $1\le j\le m-1$. Let $\zeta(y):=\hat{\xi}(z_0+\rho y)-(2-2|l|)\ln |z_0+\rho y|$ for $y\in B_1(0)$. Similar as the proof of (\ref{eq:thm:3_1:nabla:xi}), by (\ref{eq:liouville:R2}) and (\ref{eq:liouville:phi}) we have 
	\[
		- \Delta_y\zeta 
		= - \rho^2\Delta\hat{\xi}(z_0+\rho y)=\frac{8\rho^2(e^{\hat{\xi}}-1)}{(1+|z_0+\rho y|^2)^2}=O(\rho^{-2}), \quad y\in B_1(0). 
	\]
	By (\ref{eq:thm:3_1:xi:5}), we have $\zeta(y)=O(\rho^{-1})$ in $B_1(0)$. By elliptic estimate, we have $\nabla_y\zeta(y)=O(\rho^{-1})$ in $B_1(0)$. Thus 
	\begin{equation*}
		\nabla_z\hat{\xi}=(2-2|l|) \frac{z}{|z|^2}+\frac{1}{\rho}\nabla_y\zeta=(2-2|l|) \frac{z}{|z|^2}+O(\rho^{-2})=(2-2|l|) \frac{z}{|z|^2}+O(|z|^{-2}), 
	\end{equation*}
	for $z\in B_1(z_0)$. By (\ref{eq:thm:3_1:nabla:phi}), the above, the facts that $z=\frac{\sin\theta}{1-\cos\theta}(\cos\phi, \sin\phi)^T$ and $1-\cos\theta=\frac{2}{1+|z|^2}$, we have 
	\[
	\begin{split}
		\nabla_{\bS^2}\varphi & 
		= - \frac{2(1-|l|)}{1-\cos\theta}[e_{\theta}, e_{\phi}] A^T \frac{z}{|z|^2}+\frac{O(1)}{|z|^2(1-\cos\theta)}=\frac{2(|l|-1)}{\sin\theta}e_{\theta}+O(1).
	\end{split}
	\]
	Note $x-N=(\sin\theta\cos\phi, \sin\theta\sin\phi, \cos\theta-1)^T$, we have 
	\[
		\nabla_{\bS^2}\ln |x-N|=\frac{1}{2}\partial_{\theta}\ln (\sin\theta^2+(1-\cos\theta)^2)e_{\theta}=\frac{1+\cos\theta}{2\sin\theta}e_{\theta}=\frac{1}{\sin\theta}e_{\theta}+O(1).
	\]
	Thus we have 
	\[
		\nabla_{\bS^2}\varphi=\frac{2(|l|-1)}{\sin\theta}e_{\theta}+O(1)=2(|l|-1)\nabla_{\bS^2}\ln |x-N|+O(1). 
	\]
	Note $l_{m}=-l$, then (\ref{eq:thm:3_1}) for $j=m$ follows from the above. 
\end{proof}

\begin{rmk}
	For the existence and nonexistence of solutions $\varphi$ of (\ref{eq:liouville:S2}) on $\bS^2\setminus\{P_1, ..., P_m\}$ satisfying $\nabla_{\bS^2}\varphi=2(|\alpha_j|-1)\nabla_{\bS^2} \ln |x-P_j|+O(1)$ for $\alpha_j\in \R$, $1\le j\le m$, see e.g. \cite{BMM, BT, LuoTian, MZ, MP1, MP2, Troyanov} and the references therein. 
\end{rmk}

\subsubsection{Solutions that are also solutions of Euler's equation}

Consider $(-1)$-homogeneous axisymmetric no-swirl solutions of the stationary Euler's equation 
\[
\left\{
	\begin{aligned}
	& ( u \cdot \nabla ) u + \nabla p = 0, \\ 
	& \dive u=0.
	\end{aligned}
\right.
\]
Let $U_{\theta}= u_{\theta}\sin\theta$ and $y=\cos\theta$. The system on $\bS^2$ can be reduced to 
\begin{equation}
	\frac{1}{2}U^2_{\theta}=c_0+c_1y+c_2y^2, \quad u_r= \frac{d}{dy}U_{\theta},
\end{equation}
for some constants $c_0, c_1, c_2$, and the solution is given by 
\[
	u_{\theta}=\pm \frac{\sqrt{2(c_0 + c_1 \cos\theta + c_2 \cos^2\theta)}}{\sin\theta}, \quad u_r=\pm \frac{c_1+2c_2\cos\theta}{\sqrt{2(c_0 + c_1 \cos\theta + c_2 \cos^2\theta)}}, \quad u_{\phi}=0, 
\]
and $p$ is determined by (\ref{eq:liouville:p}). In particular, when the polynomial $2(c_0 + c_1 y + c_2 y^2)=(ay+b)^2$ for some $a, b\in\mathbb{R}$, the above solution is also a solution of Navier-Stokes equations (\ref{eq:NS}), where 
\[
	u_{\theta}=\frac{a\cos\theta+b}{\sin\theta}, \quad u_r=a, \quad u_{\phi}=0, \quad p=-\frac{a^2+b^2+2ab\cos\theta}{2\sin^2\theta}. 
\]
These solutions are of Type 3.

\end{document}